\documentclass[a4paper]{amsart}

\pdfoutput=1

\usepackage[utf8]{inputenc}
\usepackage[T1]{fontenc}
\usepackage{lmodern}
\usepackage{amsthm, amssymb, amsmath, amsfonts, mathrsfs}
\usepackage[colorlinks=true, pdfstartview=FitV, linkcolor=blue, citecolor=blue, urlcolor=blue,pagebackref=false]{hyperref}

\usepackage{MnSymbol} % contains notin inversed

\usepackage{esint} % contains strokedint

\usepackage{microtype}

\newtheorem{proposition}{Proposition}
\newtheorem{theorem}[proposition]{Theorem}
\newtheorem{lemma}[proposition]{Lemma}
\newtheorem{corollary}[proposition]{Corollary}

\theoremstyle{remark}

\theoremstyle{definition}

\numberwithin{equation}{section}
\numberwithin{proposition}{section}

\renewcommand{\le}{\leqslant}
\renewcommand{\ge}{\geqslant}
\renewcommand{\leq}{\leqslant}
\renewcommand{\geq}{\geqslant}
\newcommand{\ls}{\lesssim}
\newcommand{\les}{\ls}
\renewcommand{\subset}{\subseteq}
\newcommand{\mcl}{\mathcal}

\newcommand{\B}{\mathbb{B}}

\newcommand{\E}{\mathbb{E}}

\newcommand{\G}{\mathcal{G}}

\newcommand{\cA}{\mathrm{A}}

\renewcommand{\L}{\mathcal{L}}
\newcommand{\M}{\mathbf{M}}

\newcommand{\Ll}{\left}
\newcommand{\Rr}{\right}

\newcommand{\1}{\mathbf{1}}
\newcommand{\R}{\mathbb{R}}

\newcommand{\Z}{\mathbb{Z}}
\newcommand{\Zd}{{\mathbb{Z}^d}}
\renewcommand{\P}{\mathbb{P}}

\newcommand{\td}{\widetilde}
\renewcommand{\tilde}{\widetilde}
\newcommand{\eps}{\varepsilon}
\renewcommand{\d}{{\mathrm{d}}}

\newcommand{\X}{\mathbf{X}}
\newcommand{\bfeta}{\boldsymbol\eta}
\newcommand{\PP}{\mathbf{P}}
\newcommand{\EE}{\mathbf{E}}

\newcommand{\la}{\left\langle}
\newcommand{\ra}{\right\rangle}
\newcommand{\vb}{\, \big \vert \, }

\renewcommand{\fint}{\strokedint}
\newcommand{\Rd}{{\mathbb{R}^d}}

\title{Heat kernel upper bounds for interacting particle systems}
\author{Arianna Giunti, Yu Gu, Jean-Christophe Mourrat}

\address[Arianna Giunti]{Institute for Applied Mathematics, University of Bonn, Bonn, Germany}

\address[Yu Gu]{Department of Mathematics, Carnegie Mellon University, Pittsburgh, PA 15213, USA}

\address[Jean-Christophe Mourrat]{Ecole normale sup\'erieure de Lyon, CNRS, Lyon, France}

\begin{document}

\begin{abstract}

We show a diffusive upper bound on the transition probability of a tagged particle in the symmetric simple exclusion process. The proof relies on optimal spectral gap estimates for the dynamics in finite volume, which are of independent interest. We also show off-diagonal estimates of Carne-Varopoulos type.

\bigskip


\medskip


\end{abstract}
\maketitle
%
%
%
%
%
%%%%%%%%%%%%%%%%%%%%%%%%%%%%%%%%%%%%%%%%%%%%%%%%%%%%%%%%%%%%%%
%%%%%%%%%%%%%%%%%%%%%%%%%%%%%%%%%%%%%%%%%%%%%%%%%%%%%%%%%%%%%%
%
%
%
\section{Introduction}
\label{s.intro}

\subsection{Main result}

The qualitative theory of stochastic homogenization of diver\-gence-form equations was developed in the late 70's \cite{K1,Y0,PV1}. By a probabilistic representation, it is equivalent to the invariance principle for the corresponding reversible diffusion in random environment. Shortly afterwards, a strikingly general invariance principle was proved for additive functionals of reversible Markov chains \cite{kipvar}.
This result enables to show at once that a reversible random walk (or diffusion) in a random environment, and a tagged particle in a symmetric exclusion process, both rescale to Brownian motion. The recent monograph \cite{komorowski2012fluctuations} covers many further developments on this approach.

The price to pay for the breadth of this result is the difficulty to strengthen or quantify it. For instance, it was asked in \cite[Remark~1.10]{kipvar} whether a tagged particle in a symmetric exclusion process satisfies an invariance principle for almost every realization of the initial configuration (a ``quenched'' invariance principle). To this day, this question is still open.

Optimal quantitative results on the homogenization of divergence-form equations with random coefficients have only started to appear recently. We refer to \cite{GO1,GO2,GO3,GNO,GNO2,MO,AS,GNO3,AM,AKM,AKM2,GO5} for a sample of the recent work, and to \cite{AKMbook} for a monograph on the subject. Previous work focused on showing quenched invariance principles, and could ultimately cover very degenerate situations such as random walks on percolation clusters \cite{sidszn,berbis,matpia,bispre,mathieu,abdh,ads1,ads4,ads5}.

In both lines of research, one central ingredient of the proofs is a heat kernel or regularity estimate. The fact that heat kernel estimates imply a quenched invariance principle was understood early on, see \cite{osada}. Proving heat kernel bounds for degenerate environments such as percolation clusters is however a comparatively recent breakthrough \cite{matrem,bar}. We refer to \cite{biskup-survey,kum} for surveys of the topic, and to \cite{bkm,ads3,nash2,ad} for more recent contributions.

We aim to develop a comparable program for the case of a tagged particle in the symmetric exclusion process. In this paper, we show diffusive heat kernel bounds for this process. To the best of our knowledge, this is the first result of this type for an interacting particle system. Our method can be applied to more general reversible particle systems, although we choose to focus on this particular case for clarity. 

In a related direction, several works aimed at proving that certain particle systems converge to equilibrium at a polynomial rate. We refer in particular to \cite{lig91, deu94, berzeg,jlqy,lanyau,ccr} for references on this aspect.

We write $(\X_t,\bfeta_t)_{t \ge 0}$ for the joint process of the tagged particle and the symmetric simple exclusion process on $\Z^d$, $d \ge 2$, started at $(X,\eta)$. We fix the average density of particles at $\rho \in (0,1)$: under the measure $\la \cdot \, | \, X = 0 \ra_\rho$, the random variables $(\eta(x))_{x \neq 0}$ are i.i.d.\ Bernoulli with parameter $\rho$. We refer to the next section for precise definitions. Here is our main result.

\begin{theorem}[Heat kernel bound]
\label{t.main}
For every $p \ge 2$, there exists a constant $C(d,\rho,p) < \infty$ such that for every $t > 0$,
\begin{equation}  
\label{e.main}
\sum_{x \in \Zd} \la \Ll( \mathbf{P}_{(X,\eta)} \Ll[ \mathbf{X}_t = x  \Rr]  \Rr)^p  \, | \, X = 0 \ra_\rho  \le C \, {t^{\Ll(1-p\Rr)\frac d 2}}.
\end{equation}
\end{theorem}
\begin{corollary}  
\label{c.main}
For every $p \ge 2$ and $\eps > 0$, there exists a constant $C(d,\rho,p,\eps) < \infty$ such that for every $x \in \Zd$ and $t > 0$,
\begin{equation}  
\label{e.cmain}
\la \Ll( \mathbf{P}_{(X,\eta)} \Ll[ \mathbf{X}_t = x  \Rr]  \Rr)^p  \, | \, X = 0 \ra_\rho ^\frac 1 p \le C \, {t^{- \Ll(\frac d 2 - \eps\Rr)}}.
\end{equation}
\end{corollary}
We can also complement this information by an off-diagonal bound of Carne-Varopoulos type. 
\begin{theorem}[Carne-Varopoulos bound]
\label{t.carne}
There exists a constant $C(d) < \infty$ such that for every $x \in \Z^d$ and $t > 0$,
\begin{equation*}  %\label{e.}
\la  \PP_{(X,\eta)} \Ll[\X_t = x\Rr]  \, | \, X = 0 \ra_\rho\le 
\Ll|
\renewcommand{\arraystretch}{1.5}
\begin{array}{ll}
 \exp\Ll(-\frac{|x|^2}{Ct}\Rr) & \text{if } |x|\leq t, \\
  \exp\Ll( -\frac{|x|}{C}\Rr) & \text{if } |x|>t.
  \end{array}
  \Rr.
\end{equation*}
\end{theorem}
The bound obtained in Theorem~\ref{t.main} is optimal, as we now explain. By the annealed central limit theorem \cite{kipvar}, there exists a constant $c(d,\rho) > 0$ such that for every $t$ sufficiently large,
\begin{equation*}  %\label{e.}
\sum_{|x| \le \sqrt{t}} \la \PP_{(X,\eta)} \Ll[ \X_t = x \Rr] \ | \ X = 0 \ra \ge c.
\end{equation*}
Up to a redefinition of $c(d,\rho) > 0$, it thus follows by Jensen's inequality that
\begin{equation*}  %\label{e.}
\Ll(t^{-\frac d 2} \sum_{|x| \le \sqrt{t}}  \la \Ll( \mathbf{P}_{(X,\eta)} \Ll[ \mathbf{X}_t = x  \Rr]  \Rr)^p  \, | \, X = 0 \ra_\rho \Rr)^\frac 1 p \ge c t^{-\frac d 2}.
\end{equation*}
This implies the bound converse to \eqref{e.main}, up to a multiplicative constant. It also immediately yields the existence of an $x \in \Z^d$ such that the left side of \eqref{e.cmain} is bounded from below by $c t^{-\frac d 2}$, up to a redefinition of $c(d,\rho) > 0$.

\subsection{Sketch of proof for the standard heat equation}
\label{ss.standard}
Our strategy is inspired by the following argument for the relaxation of the standard heat equation. Let $u$ be the parabolic Green function with the pole at the origin, i.e.\ the decaying solution to 
\begin{equation}  \label{e.standard-heat}
\Ll\{
\begin{aligned}  %\label{}
\partial_t u &= \Delta u \ \hspace{0.5cm}\text{in } (0,+\infty) \times \R, \\
u( t=0, \cdot) &= \delta_0( \cdot)  \ \hspace{0.3cm}\text{in } \R
\end{aligned}
\Rr.
\end{equation}
Our core goal (compare with Theorem~\ref{HKbound} below) is to control the decay of monotone quantities of the form $\int_{\R^d} u^p(t,\cdot)$, for $p \ge 2$. We focus on the case $p = 2$ for simplicity, and present a robust argument, which will be adapted to the particle system, for the well-known fact that
\begin{equation}  \label{e.diff.parab}
\int_\Rd u^2(t,x) \, \d x \le C \, t^{-\frac d 2}.
\end{equation}
We give ourselves a partition of $\Rd$ into boxes of size $\ell$, and for each $x \in \Rd$, we denote by $B_\ell(x)$ the box of this partition containing $x$. We start by writing 
\begin{equation}  \label{e.decomp.Green}
\int_\Rd u^2(t, x)\, \d x \le 2 \int_\Rd \Ll( u(t,x) - \fint_{B_\ell(x)} u(t,\cdot) \Rr)^2 \, \d x + 2 \int_\Rd \Ll(  \fint_{B_\ell(x)} u(t,\cdot) \Rr)^2 \, \d x,
\end{equation}
where $\fint_{B_\ell(x)} := |B_\ell(x)|^{-1} \int_{B_\ell(x)}$ is the normalized integral. 
For the first term, Poincar\'e's inequality ensures that
\begin{equation}  \label{e.poincare}
\int_\Rd \Ll( u(t,x) - \fint_{B_\ell(x)} u(t,\cdot) \Rr)^2 \, \d x  \le C_P(d) \ell^2 \int_{\R^d} |\nabla u(t, x)|^2 \, \d x,
\end{equation}
and moreover,
\begin{equation*}  %\label{e.}
\partial_t \int_\Rd u^2(t,\cdot) = - 2 \int_{\R^d} |\nabla u(t,\cdot)|^2.
\end{equation*}
Therefore, in a time-averaged sense, the first term on the right side of \eqref{e.decomp.Green} is dominated by the left side, provided that $\ell \le c \sqrt{t}$ with $c$ sufficiently small. It therefore suffices to control the second term on the right side of \eqref{e.decomp.Green}. Since $\int_\Rd u(t,\cdot)\equiv 1$ is independent of time, we get
\begin{equation*}  %\label{e.}
\int_\Rd \Ll(  \fint_{B_\ell(x)} u(t,\cdot) \Rr)^2 \, \d x \le \int_\Rd |B_\ell(x)|^{-1} \Ll(\fint_{B_\ell(x)} u(t,\cdot) \Rr)\, \d x = |B_\ell|^{-1}.
\end{equation*}
Choosing $\ell = c \sqrt{t}$ completes our sketch of proof for \eqref{e.diff.parab}.

\subsection{Difficulties in the case of the exclusion process}
We now discuss the encountered problems, and the required modifications to the argument described above, in our context of a tagged particle in a symmetric exclusion process.

The most visible difficulties in obtaining heat kernel bounds for the tagged particle are that the environment in which the particle evolves changes over time, and that the jump rates may degenerate to zero due to the exclusion mechanism.

Optimal heat kernel estimates for degenerate dynamic environments satisfying some mild assumptions were obtained in \cite{nash2}. These results cover in particular the case of a diffusion with symmetric, possibly vanishing jump rates that depend locally on an auxiliary exclusion process at equilibrium. 

The latter process is however fundamentally different from the one we consider here. Indeed, in the situation considered in \cite{nash2}, and more generally in the context of stochastic homogenization, one can first sample the dynamic or static random environment beforehand, and \emph{then} define a diffusion with the given coefficients. In contrast, in the setting we study here, the tagged particle and the bath of all the other, untagged particles cannot be thus disentangled. There is a ``retro-action'' of the particle onto its environment, which makes the approach of \cite{nash2} inapplicable. This is the core difficulty of the problem. Mathematically, this is immediately apparent when we try to write down a differential equation analogous to \eqref{e.standard-heat} for quantities such as $\PP_{(x,\eta)}[\X_t = 0]$: there is no closed equation for this quantity if we only allow~$x$ and $t$ to vary, but not~$\eta$. Similarly, for random walks in static or dynamic random environments, quantities such as the left side of \eqref{e.diff.parab} are monotone almost surely. This is not the case in our setting, and only averaged monotone quantities will be available to us. 

\smallskip

In spite of these difficulties, we will show here how to adapt the argument exposed in the previous subsection and obtain heat kernel estimates for the tagged particle. We replace the standard Poincar\'e inequality used in \eqref{e.poincare} by spectral gap inequalities for the dynamics in finite volume. The proof of these inequalities requires some care, due to the degeneracy of the rates. Moreover, since the dynamics preserves the number of particles, these inequalities will hold only if we condition on having a fixed number of particles in the box under consideration.

In the analysis of the analogue to the first term on the right side of \eqref{e.decomp.Green}, our need to fix the number of particles in individual boxes forces the appearance of conditional measures in the analogue to the last term of \eqref{e.decomp.Green}. In other words, instead of quantities such as
$\int u(t,\cdot),$
we will have to estimate the expectation of a similar quantity with the integrand multiplied by the space-dependent densities of the conditional measures. These densities are highly singular, since they concentrate on very thin sets of fixed number of particles inside a region. We first bound these densities independently of the space variable, and then use the reversibility of the dynamics to transfer the evolution onto this density. For this term, the tagged particle is irrelevant, and we can use $L^1$ contraction in the environment variable only. We then leverage on the locality of the initial condition $f = u(0,\cdot)$ to conclude.

\subsection{Outline of the paper}
In the next section, we introduce the notation and present the general result of the form of \eqref{e.diff.parab} that we will prove, see Theorem~\ref{HKbound}. In Section~\ref{s.spectral.gap}, we show a spectral gap with optimal scaling for the joint process of the tagged particle and the exclusion process in finite volume. Section~\ref{s.hk} starts with a proof of the Carne-Varopoulos bound, from which we deduce a convenient localization property. The rest of the section then implements the strategy sketched above. 

\section{Notation and reformulation}
\label{setting}
We fix an integer $d \ge 2$. 
We say that $x,y \in \Z^d$ are neighbors, and write $x \sim y$, if $|x-y| = 1$, where $|\cdot|$ is the Euclidean distance. This turns $\Z^d$ into a graph, and we denote by $\B$ the associated set of (unoriented) edges. 
For any positive integer $\ell$, we denote by $B_\ell$ the box $\{-\ell,\ldots, \ell\}^d$, and by $\B_\ell$ the set of edges with both end-points in $B_\ell$. We let 
$$
\Omega_\ell := \Ll\{(x,\eta) \in B_\ell \times \{0,1\}^{B_\ell} \ : \ \eta(x) = 1 \Rr\},
$$
$$
\Omega := \Ll\{(x,\eta) \in \Z^d \times \{0,1\}^{\Z^d} \ : \ \eta(x) = 1 \Rr\}.
$$
For $e \in \B$ and $x \in \Z^d$, we denote
$$
x^e =
\Ll| 
\begin{array}{ll}
y & \text{if } e = \{x,y\}, \\
x & \text{if } e \nowns x.
\end{array}
\Rr.
$$
In other words, $x^e$ is the image of $x$ by the transposition between the two endpoints of the edge $e$. For $\eta \in \{0,1\}^{\Z^d}$ (or $\{0,1\}^{B_\ell}$),  $\eta^e$ is the configuration such that for every $x$, $\eta^e(x) = \eta(x^e)$. For a function $f : \Omega \to \R$ (or $\Omega_\ell \to \R$), we define $f^e(x,\eta):=f(x^e, \eta^e)$. 
We study the symmetric, simple exclusion process with a tagged particle. This is the dynamics associated with the infinitesimal generator~$\L$ formally acting on a random variable $f : \Omega  \to \R$ as
$$
\L f = \sum_{e \in \B} a_e (f^e - f),
$$
where 
\begin{equation}
\label{e.def.ae}
a_e(\eta) := \1_{\{\eta^e \neq \eta\}}.
\end{equation}
We also consider the finite-volume counterparts,
$$
\L_\ell f= \sum_{e \in \B_\ell} a_e (f^e -f),
$$
where now $f : \Omega_\ell \to \R$. The  dynamics associated with $\L_\ell$ takes place in $\Omega_\ell$ and preserves the number of particles; one can check that for every $\rho \in \{0,\ldots, |B_\ell|\}/|B_\ell|$, 
the uniform measure on the set
$$
%\Omega_{\ell,\rho}:=
\Ll\{ (x,\eta) \in \Omega_\ell \ : \ \sum_{z \in B_\ell} \eta(z) = \rho |B_\ell| \Rr\}
$$
is reversible for the dynamics (that is, the operator $\L_\ell$ is symmetric with respect to this measure).
We denote this measure by $\la \, \cdot \, \ra_{\ell,\rho}$.
With a slight abuse of notation (since we also use $\eta$ to denote a deterministic quantity), we write $(X,\eta)$ for the canonical random variable on $\Omega_\ell$ (or $\Omega$). 
For general $\rho \in [0,1]$, we understand $\la \, \cdot \, \ra_{\ell,\rho}$ to be $\la \, \cdot \, \ra_{\ell,\lfloor \rho |B_\ell| \rfloor/|B_\ell|}$. For any $x \in \Z^d$, we also define $\la \, \cdot \, |\, X = x \ra_{\rho}$ 
to be the measure under which $X = x$ almost surely (and thus $\eta(x) = 1$) and $(\eta(y))_{y \notin x}$ are independent Bernoulli random variables with parameter $\rho$. When no ambiguity occurs, we may abuse notation and write
\begin{equation*}  %\label{e.}
\langle \, \cdot \, | \, x \rangle := \langle \, \cdot \, | \, X = x \rangle_{\rho}.
\end{equation*}
 For each $A \subset \Zd$, we denote by $\mcl F(A)$
the $\sigma$-algebra generated by the random variables $(\eta(x), x \in A)$.
We extend the notion of $\mcl F(A)$-measurable random variable to functions defined on $\Omega$ or $\Omega_\ell$ as follows. A function $f : \Omega \to \R$ (resp.\ $\Omega_\ell \to \R$) is said to be $\mcl F(A)$-measurable if for every $x \in \Z^d$ (resp. $B_\ell$), the random variable $f(x,\cdot)$ is $\mcl F(A)$-measurable. 
For every $p \in [1,\infty]$ and measurable $f : \Omega \to \R$, we define
\begin{equation}
\label{e.def.lp}
\|f\|_{L^p(\rho)} := \Ll( \sum_{x \in \Z^d} \la |f|^p \, | \, X= x \ra_{\rho} \Rr)^{\frac 1 p} = \Ll( \sum_{x \in \Z^d} \la |f|^p \, | \, x \ra \Rr)^{\frac 1 p},
\end{equation} 
with the usual interpretation as a supremum if $p = \infty$. In most places, the value of~$\rho$ will be clear from the context, so that we simply write $\|f\|_p := \|f\|_{L^p(\rho)}$ and keep the dependence on $\rho$ implicit.

For an integer $r \ge 0$, we say that a function $f : \Omega \to \R$ is $B_r$-\emph{local} if
\begin{align*}  %\label{e.}
& \mbox{$f$ is $\mcl F(B_r)$-measurable, and} \\
& \mbox{for every $x \in \Zd \setminus B_r$},  \ \  f(x,\eta) = 0 .
\end{align*}
We say that a function is \emph{local} if it is $B_r$-local for some $r <+\infty$. 

For a local function $f$, we define $u: \R_+ \times \Omega \to \R$ as the unique bounded solution (see e.g.\ \cite[Theorem~4.68]{lig0} or \cite[Theorem~B.3]{lig}) to
\begin{equation}
\label{e.parab}
\Ll\{ 
\begin{array}{ll}
\partial_t u = \L u , \\
u(0,\cdot) = f(\cdot).
\end{array}
\Rr. 
\end{equation}
We may also write $P_t f(\cdot) = u_t(\cdot) = u(t,\cdot)$, where $P_t$ denotes the semigroup associated with the generator $\L$. Note that in the above expressions, the single dot represents an element of $\Omega$, which is a subset of the product space $\Z^d \times \{0,1\}^{\Zd}$. In other words, an overly scrupulous notation for $u(t,x,\eta)$ would be $u(t,(x,\eta))$.
Throughout the paper we use the notation $a\les b$ in proofs, to denote $a\leq Cb$ for some constant $C < \infty$ which may depend on some additional parameters as specified in the statement to be proved.

The main result of this paper, Theorem~\ref{t.main}, is an immediate consequence of the following estimate on monotone quantities.
\begin{theorem}\label{HKbound}
Let $\rho \in (0,1)$, $f$ be a local function, and $u_t(\cdot) = u(t,\cdot)$ be the solution to equation \eqref{e.parab}. For every $p \ge 2$, there exists a constant $C(d,\rho,f, p) < \infty$ such that for every $t > 0$,
\begin{equation}
\label{e.goal}
\|u_t\|_{L^p(\rho)}^{p} \leq C \, t^{(1-p)\frac d 2}.
\end{equation}
\end{theorem}
Recalling that the left side of \eqref{e.goal} equals to
\begin{equation*}  %\label{e.}
\sum_{x \in \Z^d} \la |u_t|^p \, | \, x \ra ,
\end{equation*}
we see that inequality \eqref{e.goal} is consistent with the idea that only those summands indexed by~$x$ in a ball of radius about $\sqrt{t}$ contribute to the sum, and that each of these summands is bounded by about $t^{-p\frac d 2}$. The constant in Theorem~\ref{HKbound} can be chosen to hold uniformly over $\rho$ bounded away from $1$.

We denote the stochastic process associated with the infinitesimal generator $\L$ by $(\X_t,\bfeta_t)_{t \ge 0}$ (see \cite{lig} for a construction), by $\PP_{(x,\eta)}$ its law starting from $(x,\eta) \in \Omega$, and by $\EE_{(x,\eta)}$ the associated expectation. This is the joint process of the tagged particle and the bath of the other, mutually indistinguishable particles. By \cite[Theorem~3.16]{lig0}, the solution to \eqref{e.parab} admits the probabilistic representation
\begin{equation}
\label{e.probab.rep}
u(t,x,\eta) = \EE_{(x,\eta)} \Ll[ f(\X_t,\bfeta_t) \Rr] .
\end{equation}

\begin{proof}[Proofs of Theorem~\ref{t.main} and Corollary~\ref{c.main} from Theorem~\ref{HKbound}]
We define the local function
\begin{equation*}  \label{e.}
f(x,\eta) := \1_{x = 0},
\end{equation*}
so that if $u$ solves \eqref{e.parab} with this choice of $f$, then by \eqref{e.probab.rep}, for every $t  > 0$,
\begin{equation*}  \label{e.}
u_t(x,\eta) = \PP_{(x,\eta)} \Ll[ \X_t=0 \Rr] .
\end{equation*}
By Theorem~\ref{HKbound}, for each $p \ge 2$, there exists a constant $C(d,\rho,p) < \infty$ such that
\begin{equation*}  %\label{e.}
\sum_{x \in \Zd} \la \Ll( \mathbf{P}_{(X,\eta)} \Ll[ \mathbf{X}_t = 0  \Rr]  \Rr)^p  \, | \, X = x \ra_\rho  \le C \, {t^{\Ll(1-p\Rr)\frac d 2}}.
\end{equation*}
Theorem~\ref{t.main} follows, since by stationarity, we have
\begin{equation*}  %\label{e.}
\la \Ll( \mathbf{P}_{(X,\eta)} \Ll[ \mathbf{X}_t = x  \Rr]  \Rr)^p  \, | \ X = 0 \ra_\rho   = \la \Ll( \mathbf{P}_{(X,\eta)} \Ll[ \mathbf{X}_t = 0 \Rr]  \Rr)^p  \, | \ X = -x \ra_\rho . 
\end{equation*}

By Jensen's inequality and Theorem~\ref{t.main}, for every $q \ge p$, we have
\begin{align*}  %\label{e.}
\la \Ll( \mathbf{P}_{(X,\eta)} \Ll[ \mathbf{X}_t = x  \Rr]  \Rr)^p  \, | \ X = 0 \ra_\rho ^\frac 1 p  
& \le \la \Ll( \mathbf{P}_{(X,\eta)} \Ll[ \mathbf{X}_t = x  \Rr]  \Rr)^q  \, | \ X = 0 \ra_\rho ^\frac 1 q  \\
& \le C(d,\rho,q) \, t^{- \Ll( 1-\frac 1 q \Rr) \frac d 2}.
\end{align*}
By choosing $q$ sufficiently large, we obtain Corollary~\ref{c.main}.
\end{proof}

\section{Spectral gap inequalities}
\label{s.spectral.gap}
In this section, we show as a first ingredient towards the proof of  Theorem \ref{HKbound} that the joint process of the tagged particle and the set of all the other (indistinguishable) particles, restricted to a box of size $\ell$, relaxes over a time scale of $\ell^2$. This takes the form of the following spectral gap inequalities.
\begin{theorem}[Spectral gap]
\label{t.spectral-gap}
For every $\rho\in (0, 1)$, there exists $C_S(d,\rho) < \infty$ which increases with respect to $\rho$ and such that for every  $\ell \in \Z_{\geq1}$ and $f : \Omega_\ell \to \R$, we have
$$
\la \Ll( f - \la f \ra_{\ell,\rho} \Rr)^2 \ra_{\ell,\rho} \le C_S \ell^2 \sum_{e \in \B_\ell} \la a_e (f^e  - f)^2 \ra_{\ell,\rho}.
$$
\end{theorem}

The proof of Theorem~\ref{t.spectral-gap} is inspired by the arguments exposed in \cite{qua}. We rely on the spectral gap of the dynamics of the $\eta$ variable alone, which was proved in \cite[Lemmas~8.2 and~8.3]{qua} and \cite[Theorem~3.1]{diaconis1993comparison}. When no tagged particle is considered, the exclusion rule becomes artificial, in the sense that the dynamics becomes identical to the Kawasaki dynamics, where particles are exchanged along edges at a constant rate.
\begin{proposition}[Spectral gap for Kawasaki dynamics \cite{qua,diaconis1993comparison}]
\label{p.spec_kawa}
There exists a constant $C_K(d) < \infty$ such that for every $\rho \in [0,1]$, $\ell \in \Z_{\geq1}$ and $f : \{0,1\}^{B_\ell} \to \R$,
$$
\la \Ll( f - \la f \ra_{\ell,\rho} \Rr)^2 \ra_{\ell,\rho} \le C_K \ell^2 \sum_{e \in \B_\ell} \la (f(\eta^e)  - f(\eta))^2 \ra_{\ell,\rho},
$$
as well as, for every $x \in \B_\ell$,
$$
\la \Ll( f - \la f \vb \eta(x) = 1 \ra_{\ell,\rho} \Rr)^2 \,\big|\, \eta(x) = 1 \ra_{\ell,\rho} \le C_K \ell^2 \sum_{\substack{e \in \B_\ell \\ e \nowns x}} \la (f(\eta^e)  - f(\eta))^2 \,|\, \eta(x) = 1  \ra_{\ell,\rho}.
$$
\end{proposition}
The crucial difference between Theorem~\ref{t.spectral-gap} and Proposition~\ref{p.spec_kawa} is that the function~$f$ in Proposition~\ref{p.spec_kawa} is a function of $\eta$ only, while the the one in Theorem~\ref{t.spectral-gap} also depends on the position of the tagged particle $X$. Note that the second part of Proposition~\ref{p.spec_kawa} relies on the fact that we only consider the case $d \ge 2$. (In fact, only the one-dimensional, nearest-neighbor case needs to be excluded.) 
\begin{proof}[Proof of Theorem~\ref{t.spectral-gap}]
We take $f : \Omega_\ell \to \R$ such that $\la f \ra_{\ell,\rho} = 0$, and write
$$
|B_\ell| \la f^2 \ra_{\ell,\rho} = \sum_{x \in B_\ell}  \la f \,|\, X=x \ra_{\ell,\rho}^2 + \sum_{x \in B_\ell} \la \Ll(f - \la f \vb X=x\ra_{\ell,\rho}\Rr)^2 \,\big|\, X=x \ra_{\ell,\rho},
$$
so the first part on the right side of the above equation represents the variation induced by the tagged particle, and the second part corresponds to the variation induced by the configurations of all other indistinguishable particles. We omit the indices $\rho, \ell$ on $\langle \,\cdot\,\rangle$ and, for an edge $b\in \B_\ell$ we define $\underline b, \overline b \in B_\ell$ as the two end-points of $b$, so that $b=(\underline b,\overline b)$. We apply the standard Poincaré inequality on $B_\ell$ to the first sum above (recalling that $\la f \ra_{\ell,\rho} = 0$), and Proposition~\ref{p.spec_kawa} to the second one, to get
\begin{equation}
\begin{aligned}\label{SG1}
|B_\ell| \langle f^2 \rangle \lesssim &\ \ell^2  \sum_{(x,y) \in \B_\ell} \biggl( \langle f( X, \eta) \,|\, X=x \rangle - \langle f(X, \eta) \,|\, X=y \rangle \biggr)^2\\
&+ \ell^2 \sum_{ x \in B_\ell} \sum_{e \in \B_\ell \atop e \nowns x} \la [ f(X,\eta^e) - f(X, \eta)]^2 \,|\, X=x \ra .
\end{aligned}
\end{equation} 
We first observe that $(x^e, \eta^e) = (x, \eta^e)$ when $x \notin e$, and in addition $\eta^e = \eta$ whenever $a_e(\eta)= 0$, so we may rewrite
\begin{align*}
\sum_{e \in \B_\ell \atop  e \nowns x} \langle [ f(X,\eta^e) - f(X, \eta)]^2 \,|\, X=x \rangle = \sum_{e \in \B_\ell \atop e \nowns x} \langle a_e(\eta)[ f^e(X,\eta) - f(X, \eta)]^2 \,|\, X=x \rangle,
\end{align*}
and thus bound the second term on the right side of \eqref{SG1} by
\begin{align}
\ell^2 \sum_{ x \in B_\ell} \sum_{e \in \B_\ell \atop  e \nowns x} \langle [ f(X,\eta^e) - f(X, \eta)]^2 \,|\, X=x \rangle
& \leq  \ell^2 \sum_{ x \in B_\ell \atop e \in \B_\ell} \langle a_e(\eta)[ f^e(X,\eta) - f(X, \eta)]^2 \,|\, X=x \rangle \notag \\
&  = |B_\ell| \ell^2 \sum_{e \in \B_\ell} \langle a_e( f^e - f)^2 \rangle.
\label{SG2}
\end{align}

We now tackle the first term on the right side of \eqref{SG1}. To lighten the notation, we sometimes write the edge $(x,y)$ as $xy$. By definition, it holds
\begin{align*}
\langle f(X, \eta) \,|\, X= y \rangle = \langle f( X^{xy}, \eta^{xy}) \,|\, X=x \rangle,
\end{align*}
so we may rewrite
\begin{equation}
\begin{aligned}\label{SG8}
&\ell^2  \sum_{(x,y) \in \B_\ell} \bigl(\langle f( X, \eta) \,|\, X=x \rangle - \langle f(X, \eta) \,|\, X=y \rangle \bigr)^2\\
 &= \ell^2  \sum_{(x,y) \in \B_\ell}\bigl( \langle f( X, \eta) - f(X^{xy}, \eta^{xy}) \,|\, X=x \rangle \bigr)^2.
\end{aligned}
\end{equation}
To conclude the proof of Theorem 1, we need to smuggle the coefficient $a$ inside the expectation on the right side of the above equation. For those configurations $(x,\eta)$ with $\eta(x)=\eta(y)=1$ and thus $a^{xy}(\eta) =0$, we want to perform a finite number of flips to exchange $x$ and $y$ in an ``admissible'' way in which we always flip an edge that connects an occupied site with an unoccupied one. In order to do so, we leverage on the presence of two empty sites at positions $z_1$ and $z_2$. We now construct the sequence of flips we will use;  this sequence will only depend on the positions of $x$, $y$, $z_1$ and $z_2$.

\begin{itemize}

\item[(1)] Recall that $x \sim y$, and let $z_1$ and $z_2$ be two holes in $\eta$, at positions distinct from $x$ and $y$. We choose a shortest non-intersecting path in $B_\ell\setminus\{x,y\}$ of the form  
\[
\tilde{x}\to \tilde{y}\to \ldots \to z_1\to \ldots \to z_2,
\]
according to some arbitrary deterministic tie-breaking rule, and in such a way that the four points $x,y,\tilde x, \tilde y$ form a unit square on a plane.

\item[(2)] We flip each edge along the path, starting from the end, until the second hole $z_2$ is next to $z_1$, then we move the two holes back together to $(\tilde{x},\tilde{y})$, so that we get a configuration near $(x,y)$ of the form 
\[
\left[\begin{array}{cc}
* & \bullet \\
\circ & \circ
\end{array}
\right],
\]
where $*$ is the tagged particle at $x$, $\bullet$ is the particle at $y$ (assuming there is one, for the purpose of graphical representation), and we have moved the holes in $z_1,z_2$ to $(\tilde{x},\tilde{y})$, denoted by $\circ$.

\item[(3)] We flip four times to obtain
\[
\left[\begin{array}{cc}
* & \bullet \\
\circ & \circ
\end{array}
\right]\mapsto
\left[\begin{array}{cc}
* & \circ \\
\circ & \bullet
\end{array}
\right]\mapsto
\left[\begin{array}{cc}
\circ & * \\
\circ & \bullet
\end{array}
\right]\mapsto
\left[\begin{array}{cc}
\circ & * \\
\bullet & \circ
\end{array}
\right]\mapsto
\left[\begin{array}{cc}
\bullet & * \\
\circ & \circ
\end{array}
\right].
\]

\item[(4)] We move the two holes at $\tilde{x},\tilde{y}$ back to $z_1,z_2$ along the path.
  
\end{itemize}

We wrote the description of the sequence of flips assuming that $\eta(z_1) = \eta(z_2) = 0$, but this only served as a guide to the explanation; for arbitrary $\eta$, we may define the same sequence of edge flips, the only difference being that the flips are no longer ``allowed'' exclusion flips. In other words, we will think of the sequence of edges selected and flipped in steps (2)-(4) above as a function of $x,y,z_1$ and $z_2$ only, but not of $\eta$. We denote it by
\begin{align}\label{SG5}
S_{x,y,z_1,z_2}:=(b_i(x,y,z_1,z_2))_{i=1}^{n}.
\end{align}
We have $n=n(x,y,z_1,z_2)\lesssim  |z_1-x|+|z_2-z_1|$. By construction, 
\begin{equation}
\label{SG6a}
( x^{b_1... b_n}, \eta^{b_1... b_n} ) = (x^{xy}, \eta^{xy}),
\end{equation}
and for every $\eta$ such that $\eta(z_1) = \eta(z_2) = 0$,
\begin{equation}
\label{SG6b}
a_{b_i}( \eta^{b_1 ... b_{i-1}} ) = 1  \text{\ \ for every $i=1,..., n$. }
\end{equation}

\medskip

For a fixed $x\sim y$, we first define the random variable $Z_1(\eta)$ to be a minimizer of the function
\begin{equation}
\label{e.def.Z1}
z_1\mapsto  \mathrm{length}(\tilde{x}\to \tilde{y}\to \ldots \to z_1),
\end{equation}
among all $z_1 \in \Zd \setminus \{x,y\}$ such that $\eta(z_1) = 0$. We then define $Z_2(\eta)$ to be a minimizer of the function
\begin{equation*}
%\label{e.def.Z2}
z_2\mapsto \mathrm{length}(Z_1(\eta)\to \ldots \to z_2),
\end{equation*}
among all $z_2 \in \Zd \setminus \{x,y,Z_1(\eta)\}$ such that $\eta(z_2) = 0$. Since $\rho < 1$, the set of candidate minimizers in both definitions are non-empty for $\ell$ sufficiently large. In both definitions, we break ties according to an arbitrary deterministic rule. We can think of an algorithm for the definition of $Z_1$ that explores each candidate $z_1 \in \Zd \setminus \{x,y\}$ sequentially, starting from the minimizer of \eqref{e.def.Z1} and going increasingly, until a candidate with $\eta(z_1) = 0$ is reached. (We simply need to make sure that the tie-breaking rule defines an ordering between the sites that have the same image through the mapping \eqref{e.def.Z1}.) A similar interpretation holds for the definition of~$Z_2$. We denote by $N_2$ the number of occupied sites thus explored until both $Z_1$ and $Z_2$ are well-defined. 

We write
\begin{align}
 \langle f( X, \eta) - f(X^{xy}, \eta^{xy}) \,|\, X=x \rangle & =  \langle \1_{\{Z_1,Z_2 \in B_{\ell}\}}\bigl( f( X, \eta) - f(X^{xy}, \eta^{xy}) \bigr) \,|\, X=x \rangle \notag\\
&  = \sum_{z_1,z_2\in B_\ell} \langle \1_{\{Z_1 =z_1,Z_2=z_2\}}\bigl( f( X, \eta) - f(X^{xy}, \eta^{xy}) \bigr) \,|\, X=x \rangle.
\label{SG7}
\end{align}
For any fixed tuple $(x,y,z_1,z_2)$, we now consider the (deterministic) set $S_{x,y,z_1,z_2}$ defined in \eqref{SG5}, and use \eqref{SG6a}-\eqref{SG6b} to write 
\begin{equation*}
\begin{aligned}
& \sum_{z_1,z_2\in B_\ell} \langle \1_{\{Z_1=z_1,Z_2=z_2\}}\bigl( f( X, \eta) - f(X^{xy}, \eta^{xy}) \bigr) \,|\, X=x \rangle \\
 &=  \sum_{z_1,z_2\in B_\ell} \la \1_{\{Z_1=z_1,Z_2=z_2\}} \sum_{i=1}^{n} a_{b_i}(\eta^{b_1... b_{i-1}} )D^{b_i}f(X^{b_1... b_{i-1}},\eta^{b_1... b_{i-1}}) \,\bigg|\, X=x \ra,
\end{aligned}
\end{equation*}
where we defined $D^bf(X, \eta):=f(X^b, \eta^b) - f(X, \eta)$; we recall that $n$ depends on $x,y,z_1,z_2$.
By the above equation and \eqref{SG7}, we rewrite \eqref{SG8} as
\[
\begin{aligned}
&\ell^2  \sum_{(x,y) \in \B_\ell} \bigl(\langle f( X, \eta) \,|\, X=x \rangle - \langle f(X, \eta) \,|\, X=y \rangle \bigr)^2\\
 = &\ell^2 \sum_{(x,y) \in \B_\ell}\biggl(\sum_{z_1,z_2 \in B_{\ell} } \la \1_{\{ Z_1=z_1,Z_2=z_2 \}} \sum_{i=1}^{n} a_{b_i}(\eta^{b_1... b_{i-1}} )D^{b_i}f(X^{b_1... b_{i-1}},\eta^{b_1... b_{i-1}})  \,\bigg|\, X=x \ra \biggr)^2.
\end{aligned}
\]
Applying H\"older's inequality first in $\langle  \,\cdot  \,| \,X= x \rangle $ and then in $\sum_{z_1,z_2\in B_{\ell} }$ yields
\begin{equation}\label{SG3}
\begin{aligned}
&\ell^2  \sum_{(x,y) \in \B_\ell} \bigl(\langle f( X, \eta) \,|\, X=x \rangle - \langle f(X, \eta) \,|\, X=y \rangle \bigr)^2 \\
&\leq \ell^2  \sum_{(x,y) \in \B_\ell}\biggl(\sum_{z_1,z_2 \in B_{\ell}}\langle \1_{\{ Z_1=z_1,Z_2=z_2 \}} \,|\, X=x \rangle^\frac12 \biggr) \\
& \times\sum_{z_1,z_2 \in B_{\ell} } \la \1_{\{Z_1=z_1,Z_2=z_2\}} \, |\, X=x\ra^\frac12\la \bigl( \sum_{i=1}^{n} a_{b_i}(\eta^{b_1... b_{i-1}} )D^{b_i}f(X^{b_1... b_{i-1}},\eta^{b_1... b_{i-1}})\bigr)^2\,\bigg|\, X=x \ra.
\end{aligned}
\end{equation}

We now estimate the probability $ \la \1_{\{Z_1=z_1,Z_2=z_2\}} \, |\, X=x\ra$. If we define 
$$
r_1:= |z_1-x|-1 \geq 0, \qquad r_2:=|z_2-z_1|-1\geq 0,
$$ 
then  by the construction of the path in (1) and the definition of $Z_1,Z_2$, there exists a constant $\tilde{c}=\tilde{c}(d) > 0 $ such that 
the total number of occupied sites around $x$ and $z_1$, which we denoted by $N_2$, satisfies $N_2\geq \tilde{c}(r_1^d+r_2^d)$. Let $N=|B_\ell|-1$ be the total number of sites except $x$, and $N_1=\lfloor\rho |B_\ell|\rfloor-1$ be the total number of particles except the tagged particle. 
Since $\langle\, \cdot\, \rangle$ is the uniform measure over $\Omega_{\ell, \rho}$, and there are already $N_2$ occupied sites around $x$ and $z_1$, it follows from Lemma~\ref{l.counting} that 
\[
\langle \1_{\{ Z_1=z_1,Z_2=z_2 \}} \,|\, X=x \rangle \leq  \biggl[\binom{N}{N_1 }\biggr]^{-1}\binom{N-N_2}{N_1-N_2}\les\sqrt{\frac{N_1(N-N_2)}{N(N_1-N_2)}} \left(\frac{N_1}{N}\right)^{N_2}.
\]
If $N_2/N\leq \rho/2$, we have 
\[
\langle \1_{\{ Z_1=z_1,Z_2=z_2 \}} \,|\, X=x \rangle\les \left(\frac{N_1}{N}\right)^{N_2}\les \rho^{\tilde{c}(r_1^d+r_2^d)};
\]
if $N_2/N>\rho/2$, we have 
\[
\langle \1_{\{ Z_1=z_1,Z_2=z_2 \}} \,|\, X=x \rangle\les \sqrt{N}\left(\frac{N_1}{N}\right)^{N_2}\les \sqrt{N}\rho^{\frac{\rho N}{4}}\rho^{\frac{\tilde{c}(r_1^d+r_2^d)}{2}} \les \rho^{\frac{\tilde{c}(r_1^d+r_2^d)}{2}}.
\]
Thus, there exists $c>0$ such that $\langle \1_{\{ Z_1=z_1,Z_2=z_2 \}} \,|\, X=x \rangle\les \rho^{c(r_1^d+r_2^d)}$,
and since
\begin{equation}\label{e.rho1}
\sum_{z_1,z_2 \in B_{\ell}}\langle \1_{\{ Z_1=z_1,Z_2=z_2 \}} \,|\, X=x \rangle^{\frac 1 2}  \les \sum_{r_1,r_2=1}^{+\infty}r_1^{d-1}r_2^{d-1}\rho^{\frac {c(r_1^d+r_2^d)}{2} } < +\infty,
\end{equation}
we estimate in \eqref{SG3}
\begin{align*}
&\ell^2  \sum_{(x,y) \in \B_\ell} \bigl(\langle f( X, \eta) \,|\, X=x \rangle - \langle f(X, \eta) \,|\, X=y \rangle \bigr)^2\\
& \lesssim \ell^2  \sum_{(x,y) \in \B_\ell}\sum_{z_1,z_2 \in B_{\ell}}\rho^{\frac{c|z_1-x|^d}{2}+\frac{c|z_2-z_1|^d}{2}} \la \bigl( \sum_{i=1}^{n} a_{b_i}(\eta^{b_1... b_{i-1}} )D^{b_i}f(X^{b_1... b_{i-1}},\eta^{b_1... b_{i-1}})\bigr)^2\,\bigg|\, X=x \ra.
\end{align*}
By applying the Cauchy-Schwarz inequality to the innermost sum and recalling that $n=n(x,y,z_1,z_2) \lesssim |z_1-x|+|z_2-z_1|$, we further obtain
\begin{equation*}
\begin{aligned}
&\ell^2 \sum_{(x,y) \in \B_\ell} \bigl(\langle f( X, \eta) \,|\, X=x \rangle - \langle f(X, \eta) \,|\, X=y \rangle \bigr)^2\\
& \lesssim \ell^2 \sum_{(x,y)\in \B_\ell}\sum_{z_1,z_2 \in B_{\ell}}(|z_1-x|+|z_2-z_1|)\rho^{\frac{c|z_1-x|^d}{2}+\frac{c|z_2-z_1|^d}{2}} \\
&\hspace{3cm}\times \sum_{i=1}^{n}\la a_{b_i}(\eta^{b_1... b_{i-1}} )\bigl(D^{b_i}f(X^{b_1... b_{i-1}},\eta^{b_1... b_{i-1}})\bigr)^2\,\bigg|\, X=x \ra.
\end{aligned}
\end{equation*}
Since the argument inside $\langle \cdot \, |\, X=x \rangle$ is non-negative, we use the crude bound
\begin{equation}\label{SG4}
\begin{aligned}
&\la a_{b_i}(\eta^{b_1... b_{i-1}} )\bigl(D^{b_i}f(X^{b_1... b_{i-1}},\eta^{b_1... b_{i-1}})\bigr)^2 \,\big|\, X=x \ra \\
 &= \la a_{b_i}(\eta )\bigl(D^{b_i}f(X,\eta)\bigr)^2 \,\big|\, X= x^{b_1...b_{i-1}} \ra \\
&\leq \sum_{\tilde x\in B_\ell} \la a_{b_i}(\eta )\bigl(D^{b_i}f(X,\eta)\bigr)^2\,\big|\, X=\tilde x \ra = | B_\ell|\la a_{b_i}\bigl(D^{b_i}f \bigr)^2\ra,
\end{aligned}
\end{equation}
where for the ``='' we used the invariance of the measure under flips. Therefore
\begin{align*}
&\ell^2 \sum_{(x,y) \in \B_\ell} \bigl(\langle f( X, \eta) \,|\, X=x \rangle - \langle f(X, \eta) \,|\, X=y \rangle \bigr)^2\\
& \lesssim  \ell^2|B_\ell| \sum_{(x,y) \in \B_\ell}\sum_{z_1,z_2 \in B_{\ell}}(|z_1-x|+|z_2-z_1|)\rho^{\frac{c|z_1-x|^d}{2}+\frac{c|z_2-z_1|^d}{2}} \sum_{i=1}^{n(x,y,z_1,z_2)} \la a_{b_i} \bigl(D^{b_i}f\bigr)^2 \ra\\
&%\stackrel{z_2 \mapsto z_2+z_1,z_1\mapsto z_1+x }{\les} 
\les
\ell^2|B_\ell|\sum_{(x,y) \in \B_\ell}\sum_{z_1,z_2 \in \Z^d}(|z_1|+|z_2|)\rho^{\frac{c|z_1|^d}{2}+\frac{c|z_2|^d}{2}}  \sum_{i=1}^{n(x,y,x+z_1,x+z_1+z_2)} \la a_{b_i} \bigl(D^{b_i}f\bigr)^2 \ra.
\end{align*}
For each $z_1,z_2$ fixed, by our construction of $S_{x,y,x+z_1,x+z_1+z_2}$ we observe that in the double sum $\sum_{(x,y) \in \B_\ell}\sum_{i=1}^{n(x,y,x+z_1,x+z_1+z_2)}$, each edge $b_i \in \B_\ell$ is repeated  $\lesssim |z_1|^d+|z_2|^d$ times. The non-negativity of the argument in $\langle \cdot \rangle$ allows us to estimate
\begin{align*}
\sum_{(x,y) \in \B_\ell}\sum_{i=1}^{n(x,y,x+z_1,x+z_1+z_2)} \la a_{b_i} \bigl(D^{b_i}f\bigr)^2 \ra\lesssim (|z_1|^d+|z_2|^d)\sum_{e \in \B_\ell}\la a_{e} \bigl(D^{e}f\bigr)^2 \ra
\end{align*}
and thus
\begin{equation}\label{SG9}
\begin{aligned}
&\ell^2 \sum_{(x,y) \in \B_\ell} \bigl(\langle f( X, \eta) | X=x \rangle - \langle f(X, \eta) \,|\, X=y \rangle \bigr)^2\\
&\lesssim \ell^2|B_\ell| \ \biggl( \sum_{z_1,z_2 \in \Z^d}(|z_1|+|z_2|)(|z_1|^d+|z_2|^d)\rho^{\frac{c|z_1|^d}{2}+\frac{c|z_2|^d}{2}}  \biggr) \sum_{e \in \B_\ell} \la a_{e} \bigl(D^{e}f\bigr)^2 \ra\\
& \lesssim\ell^2|B_\ell|\sum_{e \in \B_\ell} \la a_{e} \bigl(D^{e}f\bigr)^2 \ra.
\end{aligned}
\end{equation}
Inserting this last inequality and \eqref{SG2} into \eqref{SG1} concludes the proof of the spectral gap inequality. From \eqref{e.rho1} and \eqref{SG9}, it is clear that we can choose the constant $C_S(d,\rho)$ increasing with $\rho\in (0,1)$.
\end{proof}
%
%
%
%%%%%%%%%%%%%%%%%%%%%%%%%%%%%%%%%%%%%%%%%%%%%%%%%%%%%%%%%%%%%%
%%%%%%%%%%%%%%%%%%%%%%%%%%%%%%%%%%%%%%%%%%%%%%%%%%%%%%%%%%%%%%
%
%
%
\section{Proofs of the main results}
\label{s.hk}
The main goal of this section is to prove Theorem~\ref{HKbound}. From now on, we fix a local function $f$. Without loss of generality, we may assume that $f \ge 0$ (and therefore $u \ge 0$). In the spirit of the argument sketched in Subsection~\ref{ss.standard}, we first reduce this proof to the following bound.
\begin{proposition}\label{prop.goal}
Under the assumptions of Theorem \ref{HKbound}, for every $p \ge 1$, there exist a constant $C(d,\rho,f,p) < \infty$ and, for every $\delta > 0$, a constant $C'(\delta) < \infty$ such that for every $t \ge 0$,
\begin{align}\label{goal.p}
|| P_t f ||_{2p}^{2p}= || u_t ||_{2p}^{2p} \leq C \biggl( \delta \, t \sum_{x\in\Z^d}\sum_{e\in \B} \la a_e \bigl( (u_t^e)^p - u_t^{p} \bigr)^2 \, \big| \, x \ra + C'(\delta) t^{(1-2p) \frac d 2 }\biggr).
\end{align}
\end{proposition}
\begin{proof}[Proof of Theorem~\ref{HKbound} from Proposition~\ref{prop.goal}]
We first observe that 
\begin{equation}\label{p.derivative}
\partial_t \|u_t\|_{2p}^{2p} = - p \sum_{x\in\Z^d}\sum_{e\in \B} \la ( (u_t^e)^{2p-1} - u_{t}^{2p-1}) a_e (u_t^e - u_t) \, | \, x \ra \leq 0.
\end{equation}
We now verify that
\begin{align}\label{derivative.estimate}
 \la a_e \bigl( (u_t^e)^p - u_t^{p} \bigr)^2 \, \big| \, x \ra \leq C(p) \la ( (u_t^e)^{2p-1} - u_{t}^{2p-1}) a_e (u_t^e - u_t) \, | \, x \ra.
\end{align}
Indeed, since $u_0 = f$ is assumed to be non-negative, we have $u_t \geq 0$, and therefore the above estimate follows from the deterministic inequality 
\begin{align*}
( x^p - y^p)^{2} \leq C(p) (x^{2p-1} - y^{2p-1})(x-y), \ \ \ \ \text{ \, for $x,y \ge 0$}.
\end{align*}
In order to verify the latter, it suffices to consider the case of $x = 1$ and $y \in [0,1]$ by symmetry and homogeneity, and then the conclusion follows easily.

By \eqref{p.derivative}, the function $t \mapsto \Ll\|u_t \Rr\|_{2p}^{2p}$ is decreasing and we have, for every $t \ge0$, that
\begin{align*}  %\label{e.}
|| u_t ||_{2p}^{2p} &\leq \frac 2 t \int_{\frac t 2}^t || u_s||_{2p}^{2p} \, \d s \\
&\stackrel{\eqref{goal.p}}{\leq} C \biggl(\delta \int_{\frac t 2}^t \sum_{x\in\Z^d}\sum_{e\in \B} \la a_e \bigl( (u_s^e)^p - u_s^{p} \bigr)^2 \, \big| \, x \ra\, \d s + C'(\delta) t^{(1-2p) \frac d 2 }\biggr)\\
&\stackrel{\eqref{derivative.estimate}}{\leq}C \biggl( \delta \, C(p)  \int_{\frac t 2}^t \sum_{x\in\Z^d}\sum_{e\in \B} \la ( (u_s^e)^{2p-1} - u_{s}^{2p-1}) a_e (u_s^e - u_s) \, | \, x \ra \, \d s+  C'(\delta) t^{(1-2p) \frac d 2 }\biggr)\\
&\stackrel{\eqref{p.derivative}}{=} C \biggl( \delta \, \frac{ C(p)}{p} \bigl( || u_{\frac{t}{2}}||_{2p}^{2p} -  || u_{t}||_{2p}^{2p}  \bigr)+ C'(\delta) t^{(1-2p) \frac d 2 }\biggr)\\
&\leq C\biggl( \delta \, \frac{ C(p)}{p} || u_{\frac{t}{2}}||_{2p}^{2p} +  C'(\delta) t^{(1-2p) \frac d 2 } \biggr).
\end{align*}
It suffices now to fix $\delta$ sufficiently small such that $ C \delta \, \frac{ C(p)}{p} < 2^{(1-2p) \frac d 2 }$ to obtain Theorem~\ref{HKbound} by iteration.
\end{proof}

In the next subsection, we prove Theorem~\ref{t.carne} and derive convenient localization results for the process. We then devote the rest of the section to the proof of Proposition~\ref{prop.goal}. 

\subsection{Localization and cutoff estimate}
\label{s.local}

We start by proving Theorem~\ref{t.carne}.

\begin{proof}[Proof of Theorem~\ref{t.carne}]
Our proof is inspired by the elegant argument presented in \cite{peyre} (see also \cite{lyons,lunt}), with some modifications related to the fact that our processes are indexed by continuous time. We fix $x, y \in \Z^d$, and denote by $\xi$ the function $z \mapsto |z-x|$. We may identify $\xi$ with the function on $\Omega$ defined by $\xi(z,\eta) := \xi(z)$. The following process is a martingale:
\begin{equation*}  %\label{e.}
M_t := \xi(\X_t) - \xi(\X_0) - \int_0^t L\xi(\X_s,\bfeta_s) \, \d s.
\end{equation*}
We have
\begin{equation*}  %\label{e.}
\la \EE_{(X,\eta)} \Ll[M_t \, | \, \X_t = y \Rr] \, \big| \, x \ra = |y-x| - \la \EE_{(X,\eta)} \Ll[\int_0^t L\xi(\X_s,\bfeta_s) \, \d s \, \big| \, \X_t = y \Rr]\, \big| \, x \ra.
\end{equation*}
By reversibility, 
\begin{align*}  %\label{e.}
\la \EE_{(X,\eta)} \Ll[ M_t \, | \, \X_t = x \Rr] \, \big| \,  y \ra & = - |y-x| - \la \EE_{(X,\eta)} \Ll[\int_0^t L\xi(\X_s,\bfeta_s) \, \d s \, \big| \, \X_t = x \Rr] \, \big| \,  y \ra \\
& = - |y-x| - \la \EE_{(X,\eta)} \Ll[ \int_0^t L\xi(\X_s,\bfeta_s) \, \d s \, \big| \, \X_t = y \Rr] \, \big| \, x \ra.
\end{align*}
Combining the last two displays, we obtain
\begin{equation}  \label{e.diff.mart}
\la \EE_{(X,\eta)} \Ll[M_t \, | \, \X_t = y \Rr] \, \big| \, x \ra-\la \EE_{(X,\eta)} \Ll[ M_t \, | \, \X_t = x \Rr] \, \big| \,  y \ra = 2|y-x|.
\end{equation}
We now take a probability space with probability measure $\P$ and associated expectation $\E$ such that under $\P$, the processes $(\X_t^x, \bfeta_t^x)_{t\geq0}$ and $(\X_t^y,\bfeta_t^y)_{t\geq0}$ are independent, and are distributed according to $\la \PP_{(X,\eta)} \Ll[ \,\cdot\,\Rr] \, |\, x\ra$ and $\la \PP_{(X,\eta)} \Ll[ \,\cdot\,\Rr] \, |\, y\ra$ respectively.  
We denote the corresponding martingales by $M_t^x$ and $M_t^y$ respectively. The identity \eqref{e.diff.mart} can be rewritten as
\begin{equation*}  %\label{e.}
\E \Ll[ M_t^x - M^y_t \, | \, \X_t^x = y, \, \X_t^y = x \Rr] = 2|y-x|.
\end{equation*}
We apply Jensen's inequality to derive, for any $\lambda > 0$,
\[
\begin{aligned}
e^{\lambda|y-x|}=&\exp\left(\frac{\lambda}{2}\E \Ll[ M_t^x - M^y_t \, | \, \X_t^x = y, \, \X_t^y = x \Rr]\right)\\
\leq& \E\Ll[ e^{\frac{\lambda}{2}(M_t^x - M^y_t)} \, \big| \, \X_t^x = y, \, \X_t^y = x \Rr]\leq \frac{\E\Ll[ e^{\frac{\lambda}{2}(M_t^x - M^y_t)} \Rr]}{\P \Ll[ \X_t^x = y, \, \X_t^y = x \Rr]}.
\end{aligned}
\]
With Lemma~\ref{l.exma}, we further obtain
\begin{align}
\label{e.further}
\P \Ll[ \X_t^x = y, \, \X_t^y = x \Rr] & \le e^{-\lambda|y-x|}\E\Ll[ e^{\frac{\lambda}{2}(M_t^x - M^y_t)} \Rr]\\
\notag & \le  \exp\left[-\lambda|y-x|+Ct(e^\lambda-1-\lambda)\right],
\end{align}
for some constant $C(d)>0$. The above estimate holds for any $\lambda>0$, and we now choose $\lambda$ appropriately to minimize the right side of the above inequality. 

If $|y-x|>t$,  we have 
\[
 \exp\left[-\lambda|y-x|+Ct(e^{\lambda}-1-\lambda)\right]\leq \exp\left[-|y-x|(\lambda-C(e^{\lambda}-1-\lambda))\right].
 \]
By choosing $0<\lambda\ll1$ so that $\lambda-C(e^\lambda-1-\lambda)>0$, we find $c_1>0$ such that the right side of the above inequality is bounded by $e^{-c_1|y-x|}$ in this case. 

If $|y-x|\leq t$, by choosing $\lambda=\frac{|y-x|}{Mt}\leq M^{-1}$, we have 
 \[
 \begin{aligned}
 \exp\left[-\lambda|y-x|+Ct(e^\lambda-1-\lambda)\right]\leq& \exp(-\lambda|y-x|+C\lambda^2t e^\lambda)\\
 =&\exp\left[-\frac{|y-x|^2}{Mt}\left(1-\frac{C}{M}e^\lambda\right)\right].
 \end{aligned}
 \]
By choosing $M\gg 1$, we find $c_2>0$ such that the right side of the above inequality is bounded by $e^{-c_2|y-x|^2/t}$.
 
We finally note that, using independence and then reversibility, the left side of~\eqref{e.further} is
\begin{equation*}  %\label{e.}
\P[\X_t^x = y] \, \P[\X_t^y = x] =\Ll(  \la \PP_{(X,\eta)}\Ll[\X_t=y\Rr] \, \big| \,x\ra \Rr)^2,
\end{equation*}
and therefore the proof is complete.
\end{proof}

We now aim to show the following localization result, which says that similarly to the standard heat equation, at a fixed time $t$, we may localize the solution to $\partial_t u=\L u$ with local initial data $f$
to the box $B_L$, provided that $L\gg \sqrt{t}$.

From now on, we define $L := \lfloor \sqrt{t}\log^2 t \rfloor \vee 1$, and denote by $\cA_L$ the conditional expectation 
\begin{equation*}  %\label{e.}
\cA_L h(x,\eta) := \la h(x,\eta) \, | \, \mcl F(B_L) \ra,
\end{equation*}
with $\mcl F(B_L)$ the $\sigma$-algebra generated by the variables $(\eta(x),x \in B_L)$.
\begin{proposition}[Localization]
\label{p.localization} 
Let $h$ be a local, non-negative function, and $p \ge 1$. There exists a constant $C(d,\rho,h,p) < \infty$ such that the function $h_t=P_th$ satisfies, for every $t > 0$, 
\begin{equation}
\label{e.localization}
\|h_t\|_{2p}-\bigl(\sum_{x\in B_L} \la (\cA_L h_t)^{2p} \, | \, x\ra\bigr)^{\frac{1}{2p}}\le Ce^{-\frac{\log^2 t}{C}},
\end{equation}
where $L =\lfloor \sqrt{t} \log^2 t \rfloor \vee 1$.
\end{proposition}

Applying the above result to $u_t$ yields that, for $L = \lfloor \sqrt t \log^2 t \rfloor \vee 1$,
\begin{equation} \label{e.approx.finite2}
\begin{aligned}
\|u_t\|_{2p}\leq \bigl(\sum_{x\in B_L} \la (\cA_L u_t)^{2p} \, | \, x\ra\bigr)^{\frac{1}{2p}}+Ce^{-\frac{\log^2 t}{C}}.
\end{aligned}
\end{equation}
Therefore, in order to prove Proposition~\ref{prop.goal}, we only need to analyze the first term on the right side of \eqref{e.approx.finite2}. 

\bigskip

The rest of Subsection~\ref{s.local} is devoted to proving Proposition~\ref{p.localization}. The latter is a localization statement in two different senses: first because it replaces $h_t$ by $\cA_L h_t$; and second because it replaces a full-space sum (implicit in the norm) by one indexed by $B_L$. The second aspect of localization is obtained through the Carne-Varopoulos estimate, which indicates that the tagged particle is not super-diffusive. This information is also useful to justify the introduction of the conditioning operator $\cA_L$. The need of this conditioning in our argument is inspired by the strategy laid out for the proof of \cite[Proposition 3.1]{jlqy}. We use the Carne-Varopoulos estimate to control some boundary terms for which the tagged particle is beyond the diffusive regime. 

We start by observing that the heat kernel estimate obtained in Theorem~\ref{t.carne} implies the following bound on solutions to \eqref{e.parab}.

\begin{lemma}
\label{l.Linfinity}
Let $h$ be a local, non-negative function, and $h_t = P_t h$ be the solution to \eqref{e.parab} with initial condition $h$. There exists a constant $C(d,h)< \infty$ such that for every $t>0$ and $x\in\Z^d$, we have
\[
\la h_t \,|\, x\ra \leq  C\left(e^{-\frac{|x|^2}{Ct}}+e^{-\frac{|x|}{C}}\right).
\]
\end{lemma}

\begin{proof}
Recall the probabilistic representation \eqref{e.probab.rep}, which reads
\[
h_t(x,\eta)= \EE_{(x,\eta)}[h(\X_t,\bfeta_t)].
\]
We use the locality of $h$ to derive
\begin{align*}
\la h_t \,|\, x\ra & = \la \EE_{(X,\eta)} \Ll[ h(\X_t,\bfeta_t) \Rr] \, \big|\, x\ra
 \leq \|h\|_\infty \la \PP_{(X,\eta)} \Ll[|\X_t|\leq r_0\Rr] \, \big| \, x\ra 
\end{align*}
where $r_0$ denotes the size of the support of $h$. 
Let the function $f(t,r):=e^{-\frac{r^2}{Ct}}\1_{\{0\leq r\leq t\}}+e^{-\frac{r}{C}}\1_{\{r>t\}}$, where $C(d) < \infty$ is the constant from Theorem~\ref{t.carne}. We have that
 \[
\la h_t \, | \, x\ra\leq \|h\|_\infty \sum_{z: |x+z|\leq r_0} f(t,|z|).% \les \|h\|_\infty  l^d\max_{z: |x+z|\leq l} f(t,|z|).
 \]
 
 If $|x|\leq 2r_0$, we use the trivial bound $\la h_t \, |\, x\ra\leq \|h\|_\infty$. 
 
 If $|x|>2r_0$, $|z|$ is comparable to $|x|$, and we have 
\[
\la h_t \, |\, x\ra\leq \|h\|_\infty  \sum_{z: |x+z|\leq r_0} f(t,|z|)\leq C\|h\|_\infty  \left(e^{-\frac{|x|^2}{Ct}}+e^{-\frac{|x|}{C}}\right),
\]
with a possibly larger constant $C$. The proof is complete.
\end{proof}

In the proof of Proposition~\ref{p.localization}, we will focus on the case $p=1$, and then note that the general case $p>1$ follows directly from an $L^\infty$ bound of $h_t$. 
For positive integers $m<L$ and a sequence of increasing positive constants $\alpha_j=\exp\Ll(\frac{j}{\gamma}\Rr)$ with $\gamma=\tau\sqrt{t}$ and $\tau>0$ to be determined, we define 
\begin{equation*}
\begin{aligned}
U_{m,L,\alpha}(s)&:=\alpha_m\|\cA_mh_s\|_2^2+\sum_{k=m+1}^L \alpha_k (\|\cA_k h_s\|_2^2-\|\cA_{k-1}h_s\|_2^2)+\alpha_{L+1}(\|h_s\|_2^2-\|\cA_Lh_s\|_2^2)\\
&=\alpha_{L+1}\|h_s\|_2^2-\sum_{k=m}^L (\alpha_{k+1}-\alpha_k)\|\cA_k h_s\|_2^2.
\end{aligned}
\end{equation*}
We will first estimate $\frac{d}{ds}\|\cA_k h_s\|_2^2$ for $k\in \Z_{\geq 1}$, then derive a differential inequality for $U_{m,L,\alpha}(s)$ with $s\in [0,t]$. By Gronwall's inequality, it will lead to a bound on $U_{m,L,\alpha}(t)$ and $\|h_t\|_2^2-\|\cA_Lh_t\|_2^2$.

We define the Dirichlet energy of $h$ associated with $x\in \Z^d,e\in \B$ as
\[
D_e(h \, |\, x) :=\la a_e(h^e-h)^2 \, |\, x\ra.
\]
For every $e\in \B$ and $k\in \Z_{\geq 1}$, recall that we write $e\in \B_k$ if both end-points of $e$ belong to $B_k$. We write $e\in \partial \B_k$ if only one of these end-points belongs to $B_k$.

\begin{lemma}
\label{l.bdde}
There exists a constant $C(d,\rho,h) < \infty$ such that for any $k\in \Z_{\geq 1}$, $\beta>1$ and $s>0$, we have 
\[
\begin{aligned}
-\frac{d}{ds}\|\cA_k h_s\|_2^2\leq &\sum_{x\in\Z^d} \sum_{e\in \B_k} D_e(h_s \, |\, x)+C\beta\sum_{x\in\Z^d}\sum_{e\in \partial \B_k}D_e(h_s \, |\, x)\\
&+\frac{C}{\beta}\sum_{x\in\Z^d}\left[ \la (\cA_{k+1}h_s)^2\,|\, x\ra-\la (\cA_{k}h_s)^2\,|\, x\ra\right]+C(s^{\frac{d}{2}} e^{-\frac{k^2}{Cs}}+e^{-\frac{k}{C}}).
\end{aligned}
\]
\end{lemma}

\begin{proof}
Since $\partial_s \Ll(\cA_k h_s\Rr) = \cA_k \Ll(\partial_s h_s\Rr) = \cA_k \Ll( \L h_s \Rr)$, we have
\[
\begin{aligned}
\frac{d}{ds}\sum_{x\in \Z^d}\la (\cA_k h_s)^2 \, | \, x\ra=%&2\sum_{x\in\Z^d} \la \cA_k h_s,  Lh_s  \, | \,  x\ra\\
2\sum_{\substack{x\in\Z^d}}\sum_{e\in\B} \la \cA_k h_s, a_e(h_s^e-h_s) \, | \, x\ra.
\end{aligned}
\]
For any $e\in \B$, using the transformation $(x,\eta) \mapsto (x^e, \eta^e)$, we get
\[
\begin{aligned}
\sum_{x\in\Z^d} \la \cA_k h_s,  a_e(h_s^e-h_s)\, |\, x\ra=&\sum_{x\in\Z^d}\la (\cA_k h_s)^e,  a_e(h_s-h_s^e)\, |\, x^e\ra\\
=&\sum_{x\in \Z^d} \la (\cA_k h_s)^e,  a_e(h_s-h_s^e)\, |\, x\ra,
\end{aligned}
\]
and therefore,
\begin{equation}
\frac{d}{ds}\sum_{x\in \Z^d}\la (\cA_k h_s)^2 \, | \, x\ra=-\sum_{x\in\Z^d}\sum_{e\in\B} \la (\cA_k h_s)^e-\cA_k h_s, a_e(h_s^e-h_s)\, |\, x\ra.
\label{eq:exde}
\end{equation}
The summand on the right side of the above equation takes a similar form as the Dirichlet energy $D_e(h_s\, | \, x)$. In order to make this more precise, we distinguish between different cases of $x\in\Z^d,e\in\B$. 

(i) If $e\in \B_k$, then $a_e$ is $\mcl F(B_k)-$measurable. We also have $(\cA_kh_s)^e=\cA_k h_s^e$, so
\[
\begin{aligned}
 \la (\cA_k h_s)^e-\cA_k h_s, a_e(h_s^e-h_s)\, |\, x\ra=&\la a_e (\cA_k h_s^e-\cA_k h_s)^2\, |\, x\ra\\
 \leq & \la a_e ( h_s^e-h_s)^2\, |\, x\ra.
\end{aligned}
 \]
 
(ii) If $e\notin \B_{k+1}$ and $x\notin e$, then we have $(\cA_k h_s)^e=\cA_k h_s$, so the summands in \eqref{eq:exde} are zero.

(iii) If $e\notin B_k$ and $x\in e$, then we have $|x|\geq k$. By Lemma~\ref{l.Linfinity}, we have 
\[
\begin{aligned}
\sum_{x\in \Z^d}\sum_{e\in\B} \1_{\{e\notin \B_k, x\in e\}}|\la \cA_k h_s, a_e(h_s^e-h_s) \, | \, x\ra|\leq& C\sum_{|x|\geq k}  \left(e^{-\frac{|x|^2}{Cs}}+e^{-\frac{|x|}{C}}\right)\\
\leq & C\left(s^{\frac{d}{2}} e^{-\frac{k^2}{2Cs}}+e^{-\frac{k}{2C}}\right).
\end{aligned}
\]

(iv) If $e\in \partial \B_k$ and $x\notin e$, then by Lemma~\ref{l.boundary} we have 
\[
\begin{aligned}
&\sum_{x\in\Z^d}\sum_{e\in \B}\1_{\{e\in \partial \B_k, x\notin e\}} 
\la (\cA_k h_s)^e-\cA_k h_s, a_e(h_s^e-h_s)\, |\, x\ra\\
\leq & C\beta \sum_{x\in \Z^d}\sum_{e\in \partial \B_k} \la a_e(h_s^e-h_s)^2 \, |\, x\ra+\frac{C}{\beta} \sum_{x\in\Z^d} \left[ \la (\cA_{k+1}h_s)^2\,|\, x\ra-\la (\cA_{k}h_s)^2\,|\, x\ra\right].
\end{aligned}
\]

The proof is complete.
\end{proof}

We recall that 
\[
U_{m,L,\alpha}(s)=\alpha_{L+1}\|h_s\|_2^2-\sum_{k=m}^L (\alpha_{k+1}-\alpha_k)\|\cA_k h_s\|_2^2
\]
with $\alpha_j=\exp\Ll({\frac{j}{\gamma}}\Rr)$ and $\gamma=\tau\sqrt{t}$.

\begin{lemma}
\label{l.gronwall}
There exists $C(d,\rho,h) < \infty$ such that for any $t>1$, $s\in [0,t]$, $\tau>C$ and positive integers $m<L$, we have
\[
\frac{d}{ds}U_{m,L,\alpha}(s)\leq \frac{C}{t}U_{m,L,\alpha}(s)+C(t^{\frac{d}{2}}e^{-\frac{m^2}{Ct}}+e^{-\frac{m}{C}}).
\]
\end{lemma}

\begin{proof}
We have
\[
\frac{d}{ds}U_{m,L,\alpha}(s)=\alpha_{L+1}\frac{d}{ds}\|h_s\|_2^2-\sum_{k=m}^L(\alpha_{k+1}-\alpha_k)\frac{d}{ds}\|\cA_k h_s\|_2^2.
\]
For the first term on the right side of the above equation, we have 
\[
\frac{d}{ds}\|h_s\|_2^2= -\sum_{x\in\Z^d}\sum_{e\in \B} D_{e} (h_s \, |\, x).%-\sum_{x\in\Z^d}\sum_{e\in\B}\la a_e(h_s^e-h_s)^2 \, |\, x\ra =  -\sum_{x\in\Z^d}\sum_{e\in \B} D_{e} (h_s \, |\, x).
\]

We apply Lemma~\ref{l.bdde} to the second term to obtain
\begin{equation}
\begin{aligned}
&\frac{d}{ds}U_{m,L,\alpha}(s)\\
\leq &-\alpha_{L+1} \sum_{x\in\Z^d}\sum_{e\in \B} D_e(h_s\,|\, x)+\sum_{k=m}^L (\alpha_{k+1}-\alpha_k)\sum_{x\in \Z^d} \sum_{e\in \B_k}D_e(h_s\, |\, x)\\
&+C\beta\sum_{k=m}^L (\alpha_{k+1}-\alpha_k) \sum_{x\in\Z^d}\sum_{e\in \partial \B_k} D_e(h_s\,|\, x)+\frac{C}{\beta}\sum_{k=m}^L(\alpha_{k+1}-\alpha_k)(\|\cA_{k+1}h_s\|_2^2-\|\cA_kh_s\|_2^2)\\
&+C\sum_{k=m}^L(\alpha_{k+1}-\alpha_k)(s^{\frac{d}{2}}e^{-\frac{k^2}{Cs}}+e^{-\frac{k}{C}}),
\end{aligned}
\label{eq:diffineq}
\end{equation}
where $C=C(d,\rho,h)>0$ and $\beta>1$. We will show that by choosing $\tau$ appropriately, the total Dirichlet energy on the right side of \eqref{eq:diffineq} can be negative, and the rest is bounded up to some multiplicative constant by $U_{m,L,\alpha}$ plus some remainder term.

(i) \emph{Dirichlet energy}. Since $\beta>1$ in \eqref{eq:diffineq} is arbitrary, we choose $\beta=\sqrt{t}$. We also assume $\tau>C$ for the constant $C$ appearing in \eqref{eq:diffineq}, then 
\[
C\beta(\alpha_{k+1}-\alpha_k)<\gamma(\alpha_{k+1}-\alpha_k)\leq \alpha_{k+1},
\]
 and we have
\[
\begin{aligned}
&C\beta\sum_{k=m}^L (\alpha_{k+1}-\alpha_k) \sum_{x\in\Z^d}\sum_{e\in \partial \B_k} D_e(h_s\,|\, x)\\
\leq &\sum_{k=m}^L \alpha_{k+1}\sum_{x\in \Z^d}\left(\sum_{e\in \B_{k+1}} D_e(h_s\, |\, x)-\sum_{e\in \B_k} D_e(h_s\, |\, x)\right),
\end{aligned}
\]
which implies
\[
\begin{aligned}
&\sum_{k=m}^L (\alpha_{k+1}-\alpha_k)\sum_{x\in \Z^d} \sum_{e\in \B_k}D_e(h_s\, |\, x)+C\beta\sum_{k=m}^L (\alpha_{k+1}-\alpha_k) \sum_{x\in\Z^d}\sum_{e\in \partial \B_k} D_e(h_s\,|\, x)\\
\leq & \alpha_{L+1}\sum_{x\in\Z^d}\sum_{e\in \B_{L+1}} D_e(h_s\, |\, x)-\alpha_m \sum_{x\in\Z^d}\sum_{e\in \B_m}D_e(h_s\,|\, x).
\end{aligned}
\]
Therefore, the total Dirichlet energy on the right side of \eqref{eq:diffineq} (that is, the sum of the three first terms appearing there) is negative.

(ii) \emph{The remainder term}. Using $\gamma(\alpha_{k+1}-\alpha_k)\leq\alpha_{k+1}$, we obtain
\[
\sum_{k=m}^L(\alpha_{k+1}-\alpha_k)(s^\frac{d}{2}e^{-\frac{k^2}{Cs}}+e^{-\frac{k}{C}})\leq \sum_{k=m}^L \frac{1}{\gamma} e^{\frac{k+1}{\gamma}}s^{\frac{d}{2}}e^{-\frac{k^2}{Cs}}+\sum_{k=m}^L \frac{1}{\gamma} e^{\frac{k+1}{\gamma}}e^{-\frac{k}{C}}.
\]
For the first term, since $s<t$ and $\gamma=\tau\sqrt{t}$, we have
\[
\sum_{k=m}^L \frac{1}{\gamma} e^{\frac{k+1}{\gamma}}s^{\frac{d}{2}}e^{-\frac{k^2}{Cs}}\les  t^{\frac{d}{2}}e^{-\frac{m^2}{2Ct}} \sum_{k=m}^L \frac{1}{\sqrt{t}} e^{\frac{k+1}{\tau\sqrt{t}}} e^{-\frac{k^2}{2Ct}}\les t^{\frac{d}{2}}e^{-\frac{m^2}{2Ct}}.
\]
For the second term, we have 
\[
\sum_{k=m}^L \frac{1}{\gamma} e^{\frac{k+1}{\gamma}}e^{-\frac{k}{C}}=\frac{1}{\tau\sqrt{t}}e^{\frac{1}{\tau\sqrt{t}}} \sum_{k=m}^L e^{(\frac{1}{\tau\sqrt{t}}-\frac{1}{C})k},
\]
then we choose $\tau\geq 2C$ so $\frac{1}{\tau\sqrt{t}}-\frac{1}{C}\leq -\frac{1}{2C}$ to derive (note that $t>1$)
\[
\sum_{k=m}^L \frac{1}{\gamma}e^{\frac{k+1}{\gamma}}e^{-\frac{k}{C}}\les \sum_{k=m}^L e^{-\frac{k}{2C}}\les  e^{-\frac{m}{3C}}.
\]
Therefore, 
\[
\sum_{k=m}^L(\alpha_{k+1}-\alpha_k)(s^{\frac{d}{2}}e^{-\frac{k^2}{Cs}}+e^{-\frac{k}{C}})\leq  C(t^{\frac{d}{2}}e^{-\frac{m^2}{Ct}}+e^{-\frac{m}{C}}).
\]

Now using again the fact that $\beta(\alpha_{k+1}-\alpha_k)\leq\alpha_{k+1}$, we obtain
\[
\begin{aligned}
\frac{d}{ds}U_{m,L,\alpha}(s)\leq &\frac{C}{t}\sum_{k=m}^L \alpha_{k+1}(\|\cA_{k+1}h_s\|_2^2-\|\cA_kh_s\|_2^2)+Ct^{\frac{d}{2}}e^{-\frac{m^2}{Ct}}+Ce^{-\frac{m}{C}}\\
\leq &\frac{C}{t}U_{m,L,\alpha}(s)+C(t^{\frac{d}{2}}e^{-\frac{m^2}{Ct}}+e^{-\frac{m}{C}}).
\end{aligned}
\]
The proof is complete.
\end{proof}

We are now ready to conclude the proof of Proposition~\ref{p.localization}. 

\begin{proof}[Proof of Proposition~\ref{p.localization}]

It is clear that we only need to consider those $t\gg 1$. For such fixed $t$, we choose $m=\lfloor \sqrt{t}\log t \rfloor$ and $L=\lfloor \sqrt{t}\log^2t \rfloor$. By Lemma~\ref{l.gronwall}, we apply the Gronwall's inequality to $U_{m,L,\alpha}$ in $[0,t]$, and  derive
\[
\begin{aligned}
U_{m,L,\alpha}(t)\leq &C\left(U_{m,L,\alpha}(0)+t^{1+\frac{d}{2}}e^{-\frac{m^2}{Ct}}+te^{-\frac{m}{C}}\right)\\
\leq & C\left(U_{m,L,\alpha}(0)+t^{1+\frac{d}{2}}e^{-\frac{\log^2t}{C}}+te^{-\frac{\sqrt{t}\log t}{C}}\right).
\end{aligned}
\]
Since $h$ is a local function, it is $B_m$-local for large $t$, and recalling the definition of $U_{m,L,\alpha}$, we have
\[
U_{m,L,\alpha}(0)=\alpha_m \|\cA_mh\|_2^2=\alpha_m\|h\|_2^2.
\]
Therefore,
\[
\alpha_{L+1}(\|h_t\|_2^2-\|\cA_Lh_t\|_2^2)\leq U_{m,L,\alpha}(t)\leq C\left(\alpha_m\|h\|_2^2+t^{1+\frac{d}{2}}e^{-\frac{\log^2t}{C}}+te^{-\frac{\sqrt{t}\log t}{C}}\right),
\]
which leads to
\[
\|h_t\|_2^2-\|\cA_Lh_t\|_2^2\leq \frac{C\alpha_m}{\alpha_{L+1}}\|h\|_2^2+\frac{C}{\alpha_{L+1}}\left(t^{1+\frac{d}{2}}e^{-\frac{\log^2t}{C}}+te^{-\frac{\sqrt{t}\log t}{C}}\right)
\]
Since $\alpha_j=e^{\frac{j}{\gamma}}=e^{\frac{j}{\tau\sqrt{t}}}$, there exists a constant $C>0$ such that
\[
\|h_t\|_2^2-\|\cA_Lh_t\|_2^2\leq C e^{-\frac{\log^2 t}{C}}(1+\|h\|_2^2).
\]
This implies
\[
\begin{aligned}
\|h_t-\cA_L h_t\|_2^2=\sum_{x\in \Z^d} \la (h_t-\cA_L h_t)^2 \, |\, x\ra=&\sum_{x\in\Z^d} \la h_t^2 \, |\, x\ra-\sum_{x\in\Z^d} \la (\cA_L h_t)^2\, |\, x\ra\\
\leq &C e^{-\frac{\log^2 t}{C}}(1+\|h\|_2^2).
\end{aligned}
\]
For $p>1$, we simply use the $L^\infty$ bound $\|h_t\|_\infty \leq \|h\|_\infty$ to obtain
\[
\begin{aligned}
\|h_t-\cA_L h_t\|_{2p}^{2p}=\sum_{x\in\Z^d} \la (h_t-\cA_L h_t)^{2p} \, |\, x\ra\leq &C\sum_{x\in\Z^d} \la (h_t-\cA_L h_t)^2 \, | \, x\ra\\
\leq &C e^{-\frac{\log^2 t}{C}}(1+\|h\|_2^2).
\end{aligned}
\]

Using Lemma~\ref{l.Linfinity}, we can further restrict the tagged particle in $B_L$: 
\[
\|\cA_L h_t\|_{2p}^{2p}=\sum_{x\in B_L} \la (\cA_L h_t)^{2p} \, |\, x\ra+\sum_{x\notin B_L} \la (\cA_L h_t)^{2p} \, |\, x\ra.
\]
For the summation outside $B_L$, we have
\[
\begin{aligned}
\sum_{x\notin B_L} \la (\cA_L h_t)^{2p} \, |\, x\ra\leq C\sum_{x\notin B_L} \la \cA_L h_t \, | \, x\ra \leq &C\sum_{x\notin B_L} \left(e^{-\frac{|x|^2}{Ct}}+e^{-\frac{|x|}{C}}\right)\\
\leq & C e^{-\frac{\log^4 t}{C}}.
\end{aligned}
\]
The proof of Proposition~\ref{p.localization} is therefore complete.
\end{proof}

\subsection{Variance control: spectral gap inequality}

Recalling that in the previous step we fixed $L = \lfloor \sqrt{t} \log^2 t \rfloor \vee 1$,  we now define $\ell := \lfloor \delta \sqrt{t} \rfloor \vee 1$, for some $0<\delta\ll 1$ to be determined, and fix a partition $\{B_{\ell,i}\}_{i\in\Z_{\geq 1}}$ of $B_L$ into boxes of size $\ell$.
For each $x \in \Zd$, we denote by $B_\ell(x)$ the box of this partition to which $x$ belongs (so that $B_\ell(x)$ is \emph{not} the box centered at $x$, which we may rather denote by $x + B_\ell$). Possibly adjusting $\delta$ ever so slightly, we assume that $m := (2L+1)^d/(2\ell + 1)^d$ is an integer, i.e. we  choose $m \lesssim \delta^{-d} \log^{2d}t$ boxes of size $\ell$ partitioning $B_L$, and write $B_L=\bigcup_{i=1}^m B_{\ell,i}$.

Let $\M^{\ell}_{L} \in \Z_{\geq0}^m$ be the random vector made of the number of particles in each of the size-$\ell$ boxes partitioning $B_L$, which we decompose as
\[
\M^{\ell}_L=(M_1,\ldots,M_m),
\] 
with $M_i$ denoting the (random) number of particles in $B_{\ell,i}$.

We first show that all $M_i$ can be restricted to be in $[\frac \rho 2 |B_\ell|, \frac{\rho+1}{2} |B_\ell|] $, i.e. we only consider the cases when the number of particles in each box $B_{\ell,i}$ is relatively close to
its expectation $\rho |B_\ell|$. Define
\begin{align*}
\1^{\rho}( \M^{\ell}_L ) : = \prod_{i=1}^{m}\1_{\{\frac \rho 2 |B_\ell| \leq M_i \leq \frac{\rho+1}{2} |B_\ell|\}}.
\end{align*}
Recall that we fixed a local function $f \ge 0$, and that $u_t$ is the solution to \eqref{e.parab}.

\begin{lemma} Let $p \ge 1$. There exists a constant $C(d,\rho, f, p) < \infty$ and, for each $\delta > 0$, a constant $C'(\delta) < \infty$ such that for every $t > 0$,
\begin{align}\label{m.restricted}
\sum_{x \in B_L} \la (\cA_L u_t )^{2p} \, |\, x \ra \le  C \Ll(\sum_{x \in B_L} \la (\cA_L u_t )^{2p} \1^\rho(\M^{\ell}_L) \,  | \,x \ra + C'(\delta) t^{-100 p d } \Rr).
\end{align}
\end{lemma}
\begin{proof}
 We write
\begin{align*}
\sum_{x \in B_L} \la (\cA_L u_t )^{2p} \, | x \ra = \sum_{x \in B_L} \la (\cA_L u_t )^{2p} \1^\rho(\M^{\ell}_L) \,  | \, x \ra + \sum_{x \in B_L} \la (\cA_L u_t )^{2p} (1 - \1^\rho(\M^{\ell}_L)) \,  | \, x \ra,
\end{align*}
and bound the second term on the r.h.s by
\begin{align*}
\sum_{x \in B_L} \la (\cA_L u_t )^{2p} (1 - \1^\rho(\M^{\ell}_L)) \,  | \, x \ra \les\sum_{x \in B_L} \la 1 - \1^\rho(\M^{\ell}_L) \,  | \,x \ra.
\end{align*}
By our definition of $\1^\rho(\M^{\ell}_L)$, it holds for every $x \in B_L$ that
\begin{align*}
 \la 1 - \1^\rho(\M^{\ell}_L) \,  | \, x \ra 
 \leq \sum_{i=1}^{m} \left( \la \1_{\{M_i > \frac{\rho+1}{2}  |B_\ell|\}}  \, \big| \, x \ra+ \la \1_{\{M_i < \frac{\rho}{2}  |B_\ell|\}}  \, \big| \, x\ra \right).
\end{align*}
For each $i=1,\ldots,m$, we have 
\[
\la \1_{\{M_i > \frac{\rho+1}{2}  |B_\ell|\}}  \, \big| \, x \ra \leq \la e^{\lambda(M_i-\frac{\rho+1}{2}|B_\ell|)} \, \big| \, x\ra\leq e^\lambda\left[e^{-\frac{\rho+1}{2}\lambda}(\rho e^\lambda+1-\rho)\right]^{|B_\ell|}
\]
for any $\lambda\geq 0$, where the factor $e^\lambda$ comes from the case when $x\in B_{\ell,i}$. For the function $g(\lambda):=e^{-\frac{\rho+1}{2}\lambda}(\rho e^\lambda+1-\rho)$, it holds that $g(0)=1$ and $g'(0)=(\rho-1)/2<0$, so we may choose $\lambda=\lambda_\rho$ so that $C_\rho:=g(\lambda_\rho)<1$. Thus,  
\[
\la \1_{\{M_i > \frac{\rho+1}{2}  |B_\ell|\}}  \, \big| \, x \ra\les C_\rho^{|B_\ell|}.
\]
Since the same discussion applies to $\la \1_{\{M_i < \frac{\rho}{2}  |B_\ell|\}}  \, | \, x\ra$, we obtain 
\[
 \la 1 - \1^\rho(\M^{\ell}_L) \,  | \, x \ra\les m C_\rho^{|B_\ell|}  \les \delta^{-d}(\log^{2d}t) C_\rho^{(2\lfloor\delta \sqrt{t}\rfloor+1)^{d}},
 \]
which implies
\begin{align*}
\sum_{x \in B_L} \la (\cA_L u_t )^{2p} (1 - \1^\rho(\M^{\ell}_L)) \,  | x \ra \lesssim \frac{ C(\delta)}{t^{100 p d}},
\end{align*}
and proves \eqref{m.restricted}.
\end{proof}

Given a vector $\M^{\ell}_L$, we define for a function $h$
\begin{equation}
\label{e.def.pi}
\pi^{\ell}_{L} h(\M^{\ell}_L,x) :=|B_\ell|^{-1} \sum_{y \in B_\ell(x)} \la h \, | \, \M^{\ell}_L, y\ra.
\end{equation}
This quantity may be viewed as a local average of $h$, that is, as the expectation of $h$ conditioning on $\M^\ell_L$ and the event that the tagged particle is uniformly distributed in $B_{\ell}(x)$. Appealing to \eqref{m.restricted} and to the triangle inequality, we bound
\begin{multline}\label{e.decomposition}
\sum_{x \in B_L} \la ( \cA_L u_t)^{2p} \, | \, x \ra \\
 \lesssim \sum_{x \in B_L} \la | \cA_L u_t - \pi^\ell_L u_t|^{2p} \1^\rho(\M^{\ell}_L) \, | \, x \ra +  \sum_{x \in B_L} \la (\pi^\ell_L u_t)^{2p} \1^\rho(\M^{\ell}_L) \, | \, x \ra + C(\delta) t^{-100 p d}.
\end{multline}
We now apply the spectral gap inequality of Theorem~\ref{t.spectral-gap} to control the first term on the right side of the above display. 
\begin{proposition}\label{p.second.spectral.gap}
Let $p \ge 1$. There exists a constant $C(d,\rho,p) < \infty$ such that for every $\mcl F(B_L)$-measurable, bounded non-negative function $h : \Omega \to \R$ and $t > 0$, we have
\begin{equation}\label{s.s.g.1}
\sum_{x \in B_L} \la | h - \pi^{\ell}_{L}h(\M^{\ell}_{L},x) |^{2p} \1^{\rho}(\M_L^\ell)   \, | \,   x \ra \\
\le C \ell^2 \sum_{x\in B_L}\sum_{e\in\B_L} \la a_e((h^e)^p - h^p)^2 \, | \, x \ra.
\end{equation}
\end{proposition}

\begin{proof}
We write
\begin{align*}
\sum_{x \in B_L} \la | h - \pi^{\ell}_{L}h(\M^{\ell}_{L},x) |^{2p}  \1^{\rho}(\M_L^\ell)  \,|\,   x \ra = \sum_{i=1}^{m} \sum_{x\in B_{\ell,i}}  \la | h - \pi^{\ell}_{L}h(\M^{\ell}_{L},x) |^{2p}  \1^{\rho}(\M_L^\ell)  \,|\,   x \ra.
\end{align*}
It suffices to show that for every $i\in \{1,\ldots, m \}$,
\begin{align}\label{SG2.1}
\sum_{x \in B_{\ell,i}}  \la | h - \pi^{\ell}_{L}h(\M^{\ell}_{L},x) |^{2p}  \1^{\rho}(\M_L^\ell)  \,|\,   x \ra \lesssim \ell^2 \sum_{x\in B_{\ell,i}} \sum_{e \in \B_L} \la a_e ( (h^e)^p - h^p)^2 \,|\, x \ra.
\end{align}
Since the ordering of the partition $(B_{\ell,i})_{i = 1}^m$ is arbitrary, it suffices to prove \eqref{SG2.1} for $i = 1$.
Recalling that $\M^{\ell}_{L} = (M_1,\ldots,M_m)$, we define a decreasing sequence of $\sigma$-algebras $\{\mathcal{G}_j \}_{j=1}^m$ by
\begin{align}%\label{SG2.3}
\mathcal{G}_j:= \sigma\Ll( M_1, \ldots, M_j,   \Ll\{ \eta(\tilde x) \ : \ \tilde x \in B_L\setminus \cup_{k=1}^j B_{\ell,k}  \Rr\}  \Rr)
\end{align}
and  the following random variables for $q\geq 1$:
\begin{align}\label{SG2.7}
(h^q)_j^y:= \langle h^q \, | \, \mathcal{G}_j, y \rangle, \qquad   H^q_j := |B_{\ell}|^{-1}\sum_{ y\in B_{\ell,1}}(h^q)_j^y .
\end{align}
It is clear that  $H^q_j$ may be viewed as the expectation of $h^q$ conditioning on $\mathcal{G}_j$ and the event that the tagged particle is uniformly distributed in $B_{\ell,1}$.

\medskip

With the above notations, we write 
\[
\begin{aligned}
\sum_{x \in B_{\ell,1}}  \la | h - \pi^{\ell}_{L}h(\M^{\ell}_{L},x) |^{2p}  \1^{\rho}(\M_L^\ell)  \,|\,   x \ra
%= &\sum_{x \in B_{\ell,1}}  \langle | h - |B_\ell|^{-1}\sum_{y\in B_{\ell,1}}\langle h \ | \mathcal{G}_m, y \rangle |^{2p}  \1^{\rho}(\M_L^\ell)  \,|\,   x \rangle \\
%=& \sum_{x \in B_{\ell,1}}  \langle \langle  | h - |B_\ell|^{-1}\sum_{y\in B_{\ell,1}}\langle h \ | \mathcal{G}_m, y \rangle |^{2p} \ | \ \mathcal{G}_m , x \rangle  \1^{\rho}(\M_L^\ell)  \,|\,   x \rangle.
= \sum_{x\in B_{\ell,1}} \la \langle  | h - H^1_m |^{2p} \, | \, \mathcal{G}_m , x \rangle  \1^{\rho}(\M_L^\ell)  \,|\,   x \ra.
\end{aligned}
\]
We observe that for each $x \in B_{\ell,1}$, the random variable 
\begin{equation*}  
\la | h - H^1_m |^{2p} \, | \, \mathcal{G}_m , x \ra  \1^{\rho}(\M_L^\ell) 
\end{equation*}
depends only on $M_1$ and $(\eta(\td x),\td x \notin B_{\ell,1})$, thus we may substitute the outer measure $\langle \cdot \, | \, x \rangle$ with $\langle \cdot \, | \, x_0 \rangle$ for any fixed $x_0 \in B_{\ell, 1}$ and move the summation inside to write
\[
\sum_{x \in B_{\ell,1}}  \la |h - \pi^{\ell}_{L}h(\M^{\ell}_{L},x) |^{2p}  \1^{\rho}(\M^\ell_L)  \,|\,   x \ra =  \la \sum_{x \in B_{\ell,1}} \langle | h - H^1_m |^{2p} \, | \, \mathcal{G}_m , x \rangle  \1^{\rho}(\M^\ell_L)  \,\bigg|\,   x_0 \ra.
\]
We apply the moment inequality of Lemma \ref{l.moment} to derive
\begin{align*}
\sum_{x \in B_{\ell,1}} \la  | h - H^1_m |^{2p} \, | \, \mathcal{G}_m , x \ra \les\sum_{x \in B_{\ell,1}} \la  ( h^p - H^p_m)^{2} \, | \, \mathcal{G}_m , x \ra.
\end{align*}
It thus remains to prove \eqref{SG2.1} with the left side replaced by
\begin{equation}
\label{SG2.9}
 \la \sum_{x \in B_{\ell,1}} \la  ( h^p - H^p_m)^{2} \, | \, \mathcal{G}_m , x \ra  \1^{\rho}(\M_L^\ell)  \,\bigg|\,   x_0 \ra.
\end{equation}
For the summation in \eqref{SG2.9}, we have
\begin{align}\label{SG2.2}
\sum_{x \in B_{\ell,1}} \la  \bigl( h^p - H^p_m \bigr)^{2} \, | \, \mathcal{G}_m , x \ra = \sum_{x \in B_{\ell,1}} \la  \bigl( h^p - H^p_1 + \sum_{i=1}^{m-1} (H^p_i -  H^p_{i+1} )\bigr)^{2} \, \bigg| \, \mathcal{G}_m , x \ra \notag\\
= \sum_{x \in B_{\ell,1}} \la  \bigl( h^p - H^p_1\bigr)^{2} \, | \, \mathcal{G}_m , x \ra + \sum_{i=1}^{m-1}\sum_{x \in B_{\ell,1}} \la   (H^p_i -  H^p_{i+1} )^{2} \, | \, \mathcal{G}_m , x \ra.
\end{align}

\medskip

We start by observing that the first term on the right side of \eqref{SG2.2} can be rewritten as
\begin{align}%\label{SG2.3}
\sum_{x \in B_{\ell,1}} \la  \bigl( h^p - H^p_1\bigr)^{2} \, | \, \mathcal{G}_m , x \ra 
&= \la \sum_{x \in B_{\ell,1}} \la \bigl( h^p - H^p_1 \bigr )^{2} \, | \,  \mathcal{G}_1 , x \ra \, \bigg| \,  \mathcal{G}_m , x_0 \ra.
\end{align}
By applying to the term inside $\langle \cdot \, |\, \mathcal{G}_m , x_0 \rangle$ the spectral gap inequality of Theorem~\ref{t.spectral-gap} in the box $B_{\ell,1}$ and with density given by $\rho_1:=M_1/|B_\ell|\in[\frac{\rho}{2},\frac{\rho+1}{2}]$, we obtain 
\begin{equation}
\begin{aligned}
\label{SG2.6}
&\sum_{x \in B_{\ell,1}} \la  \bigl( h^p - H^p_1\bigr)^{2} \ | \ \mathcal{G}_m , x \ra \\
\leq &C_S(\rho_1,d)\ell^2 \la \sum_{x \in B_{\ell,1}}\sum_{e\in \mathbb{B}_{\ell,1} } \la a_e \bigl( (h^e)^p - h^p \bigr)^{2} \, | \, \mathcal{G}_1 , x \ra \, \bigg| \, \mathcal{G}_m , x_0 \ra \\
= &C_S(\rho_1, d)\ell^2\sum_{x \in B_{\ell,1}}\sum_{e \in \mathbb{B}_{\ell,1} } \la a_e \bigl( (h^e)^p - h^p \bigr)^{2}  \, | \, \mathcal{G}_m , x \ra.
\end{aligned}
\end{equation}
To deal with the second term on the right side of \eqref{SG2.2}, the idea is similar. For each $i=1,\ldots,m-1$, we note that $\la   (H^p_i -  H^p_{i+1} )^{2} \, | \, \mathcal{G}_m , x \ra%=\sum_{x\in B_{\ell,1}} \la
$ does not depend on $x$, and 
\[
\la   (H^p_i -  H^p_{i+1} )^{2} \, | \, \mathcal{G}_m , x \ra= \la  \la (H^p_i -  H^p_{i+1} )^{2} \, | \, \mathcal{G}_{i+1},y\ra \, | \, \mathcal{G}_m,y  \ra
\]
for any $y\in B_{\ell,1}$. We also have 
\[
\begin{aligned}
(H^p_i -  H^p_{i+1} )^{2}=&|B_\ell|^{-2}\bigl(\sum_{y\in B_{\ell,1}} \bigl(\la h^p \, |\, \G_i,y\ra- \la h^p \, |\, \G_{i+1},y\ra\bigr)\bigr)^2\\
\leq & |B_\ell|^{-1} \sum_{y\in B_{\ell,1}}\bigl(\la h^p \, |\, \G_i,y\ra- \la h^p \, |\, \G_{i+1},y\ra\bigr)^2.
\end{aligned}
\]
After conditioning on $\G_{i+1}$, we apply the spectral gap inequality of Proposition~\ref{p.spec_kawa} to the box $B_{\ell,i+1}$ and derive for every $y\in B_{\ell,1}$ that
\[
\la \bigl(\la h^p \, |\, \G_i,y\ra- \la h^p \, |\, \G_{i+1},y\ra\bigr)^2 \, \big|\, \G_{i+1},y \ra \leq C_K \ell^2 \sum_{e\in \B_{\ell,i+1}} \la a_e \bigl( (h^e)^p - h^p \bigr)^{2} \,\big|\, \G_{i+1},y\ra,
\]
and this implies 
 \begin{multline}\label{SGnew}
 \sum_{i=1}^{m-1}\sum_{x \in B_{\ell,1}} \la   (H^p_i -  H^p_{i+1} )^{2} \, | \, \mathcal{G}_m , x \ra\\
 \les C_K \ell^2 \sum_{i=1}^{m-1}\sum_{y\in B_{\ell,1}}\sum_{e\in \B_{\ell,i+1}} \la a_e \bigl( (h^e)^p - h^p \bigr)^{2} \,\big|\, \G_{m},y\ra.
\end{multline} 
 Combining \eqref{SG2.6} and \eqref{SGnew}, we obtain
\begin{multline*}
\sum_{x \in B_{\ell,1}} \la  \bigl( h^p - H^p_m \bigr)^{2} \, \big| \, \mathcal{G}_m , x \ra  \\
\lesssim   \, \Ll(C_S(\rho_1,d)\vee C_K \Rr)\, \ell^2 \sum_{x \in B_{\ell,1}}\sum_{e\in \B_L} \la a_e \bigl( (h^e)^p - h^p \bigr)^{2} \, | \,  \mathcal{G}_m,  x \ra.
\end{multline*}
We finally plug this inside \eqref{SG2.9} and obtain \eqref{SG2.1} with $i=1$. We need the factor $\1^\rho(\M_L^\ell)$ to bound $C_S(\rho_1, d) \leq C_S(\frac{\rho+1}{2},d)$. 
\end{proof}
By Proposition~\ref{p.second.spectral.gap} and the fact that $\pi^\ell_L \cA_L = \pi^\ell_L$, $\ell=\lfloor \delta \sqrt{t} \rfloor$, we can therefore reduce  \eqref{e.decomposition} to
\begin{multline}\label{e.decomposition2}
\sum_{x \in B_L} \la ( \cA_L u_t)^{2p} \, | x \ra \\
 \lesssim \delta^2 t \sum_{x\in B_L}\sum_{e\in\B_L} \la a_e((u_t^e)^p - u_t^p)^2 \, | \, x \ra +  \sum_{x \in B_L} \la (\pi^\ell_L u_t)^{2p} \1^\rho(\M_L^\ell) \, | \, x \ra + C(\delta) t^{-100 p d}.
\end{multline}

\subsection{Conclusion}

Summarizing, it follows from Proposition~\ref{p.localization} and \eqref{e.decomposition2} that for $\ell = \lfloor \delta \sqrt{t} \rfloor \vee 1$ and $L = \lfloor \sqrt{t} \log^2 t \rfloor \vee 1$,
\begin{multline} \label{e.intermediate}
|| u_t ||_{2p}^{2p}  \lesssim \delta^2 \, t\sum_{x\in\Z^d}\sum_{e\in\B} \la a_e( (u_{t}^e)^p - (u_t)^p)^2 \, | \, x \ra \\
 + \sum_{x \in B_L} \la \Ll(\pi^{\ell}_{L} u_{t}(\M^{\ell}_{L},x) \Rr)^{2p} \1^\rho(\M_L^\ell)  \,\big|\,   x \ra + C(\delta ) t^{(1-2p)\frac d 2}.
\end{multline}
In order to complete the proof of Proposition~\ref{prop.goal}, it therefore suffices to show the following.
\begin{proposition}\label{p.entropy}
There exists a constant $C'(\delta) = C'(d,\rho,f,p,\delta) < \infty$ such that for every $t > 0$, we have
\[
\sum_{x \in B_L} \la \bigl( \pi^{\ell}_{L}  u_t(\M^{\ell}_{L},x) \bigr)^{2p}  \1^\rho(\M^\ell_L)  \, \big| \,   x \ra \leq C'(\delta) \,  t^{ (1 - 2p) \frac d 2}.
\]
\end{proposition}

\medskip

From now on, we denote by $\langle \,\cdot\, \rangle$ the ``pure'' Kawasaki measure on the lattice~$\Zd$, i.e. the product measure of independent Bernoulli $\{\eta(y)\}_{y \in \Zd}$ with parameter $\rho$.
We also define the operator $\L_K$ associated to the Kawasaki dynamic acting on a random variable $f=f(\eta)$ as
\begin{align}\label{pure.kawasaki}
\L_K f(\eta) := \sum_{e \in \B} ( f(\eta^e) - f(\eta) ).
\end{align}
For every $x \in B_L$, we denote with $h^M(x, \cdot)$ the Radon-Nikodym derivative of the measure $\langle \,\cdot \, | \, \M^\ell_L, x \rangle$ with respect to $\langle \,\cdot \, | \, x \rangle$, i.e. for $g \in L^1(\Omega)$ we have
\begin{align}\label{R-D}
\langle g(x ,\cdot ) \, |\, \M_L^\ell=M, x \rangle = \langle g(x, \cdot ) h^M(x,\cdot )  \, | \, x \rangle. 
\end{align}
Analogously, we denote by $\tilde h^M $ the Radon-Nikodym derivative of $\langle \,\cdot \, | \, \M_L^\ell \rangle$ with respect to $\langle \,\cdot\, \rangle$. We have the following lemma.
\begin{lemma}\label{projection.kawasaki}
Let the vector $M=(M_1,\ldots,M_m)$ be fixed and such that $M_i \geq 1$ for all $i= 1, ..., m$.
For any $x \in B_L$, we define $i(x)$  such that $B_{\ell, i(x)} = B_\ell(x)$. We have
\begin{equation}\label{R-D.identity}
h^{M}(x, \eta ) =  \frac{\rho| B_\ell |}{M_{i(x)}} \tilde h^M( \eta ),
\end{equation}
in the sense that for every random variable $g= g( X, \eta)$, we have
\begin{align*}
\langle  g(X, \cdot ) \,  |\, \M_L^\ell=M , x \rangle =  \frac{\rho| B_\ell |}{M_{i(x)}} \langle  g(X, \cdot ) \, \tilde h^M( \cdot ) \, | \, x  \rangle.
\end{align*}
\end{lemma}
\begin{proof}
Let  $\1_{M}$ be the indicator function of the event $\M_L^\ell=M$ (i.e. of having $M_i$ particles in each of the boxes $B_{\ell, i}$ partitioning $B_L$), let $g= g(X, \eta)$ and $\tilde g = \tilde g( \eta)$. 
Since we may write
\begin{align*}
\langle  \tilde g( \cdot ) \,|\, \M_L^\ell=M \rangle =\langle  \tilde g( \cdot ) \, \tilde h^M( \cdot)  \rangle = \la \tilde g( \cdot ) \, \frac{\1_{M}( \cdot) }{\langle \1_M \rangle }  \ra,
\end{align*}
and for every $x \in B_L$
\begin{align*}
\langle  g(X, \cdot) \, | \, \M_L^\ell=M ,  x \rangle = \langle g(x, \cdot) h^M(x, \cdot) \ | \  x \rangle = \la  g(x, \cdot)  \frac{ \1_{\{ \eta(x)= 1\}}(\cdot) \1_{M}(\cdot) }{ \langle \1_{\{ \eta(x)= 1\}} \1_{M} \rangle }   \ra,
\end{align*}
it holds that
\[
\begin{aligned}
\langle  g(X, \cdot) \, | \, \M_L^\ell=M ,  x \rangle  & = \frac{\langle \1_M  \rangle}{ \langle \1_{\{ \eta(x)= 1\}} \1_{M} \rangle }  \langle  g(x, \cdot)  \1_{\{ \eta(x)= 1\}}(\cdot)\tilde h^M(\cdot)\rangle \\
&=\rho \frac{\langle  \1_M  \rangle}{ \langle \1_{\{ \eta(x)= 1\}} \1_{M} \rangle }  \langle g(X, \cdot) \tilde h^M(\cdot) \, | \,  x  \rangle.
\end{aligned}
\]
We establish identity \eqref{R-D.identity} by observing that by the independence of each $\eta(y)$ and the construction of the vector $\M_L^\ell$, we have
\begin{equation*}
\frac{\langle  \1_M  \rangle}{ \langle \1_{\{ \eta(x)= 1\}} \1_{M} \rangle }  =  \frac{| B_\ell |}{M_{i(x)}}. \qedhere
\end{equation*}
\end{proof}

\begin{proof}[Proof of Proposition~\ref{p.entropy}]
For any $x \in B_L$, by the definition of $h^M(x, \cdot)$ through \eqref{R-D}, we rewrite
\[
\begin{aligned}
\sum_{x \in B_L} \la \bigl(\pi^{\ell}_{L} u_{t}(\M^{\ell}_{L},x) \bigr)^{2p}   \1^\rho(\M_L^\ell)  \,|\,   x \ra=& \sum_{x \in B_L} \la \bigl(|B_\ell|^{-1} \sum_{y\in B_\ell(x)} \la  u_{t} \, | \, \M_L^\ell, y  \ra \bigr)^{2p}   \1^\rho(\M_L^\ell)  \,\bigg|\,   x \ra \\
=&|B_\ell|^{-2p} \sum_{x \in B_L} \la \bigl(\sum_{y\in B_\ell(x)} \la u_th^M   \, | \,  y \ra \bigr)^{2p}  \1^\rho(\M_L^\ell)   \,\bigg|\,   x \ra,
\end{aligned}
\]
where here and in the following, we interpret $\langle u_th^{M}   \, | \,  y\rangle$ and $\langle u_t\tilde{h}^{M}   \, | \,  y\rangle$ as random variables with $M=\M_L^\ell$. 
By our restriction on the values of $\M_L^\ell$ given by the random variable $\1^\rho(\M_L^\ell)$, we can appeal to Lemma \ref{projection.kawasaki} to derive
\begin{align*}
&\sum_{x \in B_L} \la \bigl(\pi^{\ell}_{L} u_{t}(\M^{\ell}_{L},x) \bigr)^{2p}   \1^\rho(\M_L^\ell)  \,|\,   x \ra\\
&= \rho^{2p} |B_\ell|^{-2p} \sum_{x \in B_L} \la \frac{| B_\ell |^{2p}}{M_{i(x)}^{2p}}
\biggl(\sum_{y\in B_\ell(x)} \la  u_{t}(y, \eta)\tilde h^M(\eta) \, | \, y  \ra \biggr)^{2p}  \1^\rho(\M_L^\ell)   \,\bigg|\,   x \ra.
\end{align*}
We bound the above term by
\[
|B_\ell|^{-2p} \sum_{x \in B_L} \la
\biggl(\sum_{y\in B_\ell(x)} \la  u_{t}(y, \eta)\tilde h^M(\eta) \, | \, y  \ra \biggr)^{2p}\1^\rho(\M_L^\ell)    \,\bigg|\,   x \ra.
\]
For fixed $i$, the above expectation is independent of $x\in B_{\ell,i}$, so we write for an arbitrary $x_i\in B_{\ell,i}$
\begin{equation}\label{e.entropy2}
\begin{aligned}
&|B_\ell|^{-2p} \sum_{x \in B_L} \la
\biggl(\sum_{y\in B_\ell(x)} \la  u_{t}(y, \eta)\tilde h^M(\eta) \, | \, y  \ra \biggr)^{2p}\1^\rho(\M_L^\ell)    \,\bigg|\,   x \ra\\
&=|B_\ell|^{-2p+1} \sum_{i=1}^m \la
\biggl(\sum_{y\in B_{\ell,i}} \la  u_{t}(y, \eta)\tilde h^M(\eta) \, | \, y  \ra \biggr)^{2p}\1^\rho(\M_L^\ell)    \,\bigg|\,   x_i \ra\\
&=|B_\ell|^{-2p+1} \sum_{i=1}^m \la g_i(\M_L^\ell)\1^\rho(\M_L^\ell)  \,|\,x_i\ra,
\end{aligned}
\end{equation}
with 
\[
g_i(\M_L^\ell):=\biggl(\sum_{y\in B_{\ell,i}} \la  u_{t}(y, \eta)\tilde h^M(\eta) \, | \, y  \ra \biggr)^{2p}.
\]
 We claim that \eqref{e.entropy2} can be bounded by 
\[
|B_\ell|^{-2p+1} \sum_{i=1}^m \la g_i(\M_L^\ell)\1^\rho(\M_L^\ell)  \,|\,x_0\ra= \la |B_\ell|^{-2p+1}\sum_{i=1}^m g_i(\M_L^\ell)\1^\rho(\M_L^\ell)  \,\bigg|\, x_0\ra
\]
for an arbitrary $x_0\notin B_L$. Consider $\la g_i(\M_L^\ell)\1^\rho(\M_L^\ell)  \,|\,x_i\ra$ for any $i$: it can be written as
\[
\la g_i(\M_L^\ell)\1^\rho(\M_L^\ell) \,|\,x_i\ra=\sum_{N_1,\ldots,N_m\in \Z_{\geq0}} g_i(N_1,\ldots,N_m)\1^\rho(N_1,\ldots,N_m) \prod_{j=1}^m \la \1_{\{M_j=N_j\}} \,|\, x_i\ra.
\]
For $j\neq i$, we have $ \la \1_{\{M_j=N_j\}} \,|\, x_i\ra= \la \1_{\{M_j=N_j\}} \,|\, x_0\ra$, and since $g_i\geq 0$, it suffices to show that
\[
 \la \1_{\{M_i=N_i\}} \,|\, x_i\ra\les  \la \1_{\{M_i=N_i\}} \,|\, x_0\ra,
 \]
 which is equivalent with
 \[
\binom{|B_\ell|-1}{N_i-1}\rho^{N_i-1}(1-\rho)^{|B_\ell|-N_i}\les  \binom{|B_\ell|}{N_i}\rho^{N_i}(1-\rho)^{|B_\ell|-N_i}.
 \]
 Using the simple estimate
\[
|B_\ell|^{-2p+1}\sum_{i=1}^m g_i(\M_L^\ell)\1^\rho(\M_L^\ell) \leq |B_\ell|^{-2p+1}\biggl(\sum_{y\in \Z^d} \la  u_{t}(y, \eta)\tilde h^M(\eta) \, | \, y  \ra \biggr)^{2p},
%|B_\ell|^{-2p+1}   \biggl(\sum_{i=1}^m
%\sum_{y\in B_{\ell,i}} \la  u_{t}(y, \eta)\tilde h^M(\eta) \, | \, y  \ra \biggr)^{2p}\leq |B_\ell|^{-2p+1}\biggl(\sum_{y\in \Z^d} \la  u_{t}(y, \eta)\tilde h^M(\eta) \, | \, y  \ra \biggr)^{2p}.
\]
we have proved
\begin{equation}\label{e.entropy3}
\sum_{x \in B_L} \la \bigl(\pi^{\ell}_{L} u_{t}(\M^{\ell}_{L},x) \bigr)^{2p}   \1^\rho(\M_L^\ell)  \,|\,   x \ra \les  |B_\ell|^{-2p+1} \la \biggl(\sum_{y\in \Z^d} \la  u_{t}(y, \eta)\tilde h^M(\eta) \, | \, y  \ra \biggr)^{2p} \,\bigg|\,x_0\ra.
\end{equation}

We remark that the term  $\sum_{y\in \Z^d} \la  u_{t}(y, \eta)\tilde h^M(\eta) \, | \, y  \ra$ in the r.h.s. above depends on the variables with respect to which we take the outer expectation $\langle \, \, \cdot \, \,  | \, x_0 \rangle$, only through the vector $M$. In other words, if we denote with $ \tilde\eta$ the configuration with respect to which the measure $\langle \, \, \cdot \, \,  | \, x_0 \rangle$ is defined, then we have that $M = M(\tilde \eta)$ and that the r.h.s. in \eqref{e.entropy3} may be rewritten as
\begin{align*}
|B_\ell|^{-2p+1} \la \biggl(\sum_{y\in \Z^d} \la  u_{t}(y, \eta)\tilde h^{M(\tilde \eta)}(\eta) \, | \, y  \ra_{\eta} \biggr)^{2p} \,\bigg| \,x_0\ra_{\tilde\eta}.
\end{align*}
Here, with $\eta$ or $\tilde\eta$ subscript in the expectations we stress the variable with respect to which each expectation is taken. Therefore, it follows that for every fixed $\tilde\eta$, and accordingly fixed $M= M(\tilde\eta)$, we may apply to the term $\sum_{y\in \Z^d} \la  u_{t}(y, \eta)\tilde h^M(\eta) \, | \, y  \ra_\eta$ reversibility and rewrite it as
\begin{align*}
\sum_{y\in \Z^d} \la  u_{t}(y, \eta)\tilde h^M(\eta) \, | \, y  \ra_\eta = \sum_{y\in \Z^d} \la  f(y, \eta) P_t\tilde h^M(\eta) \, | \, y  \ra_\eta.
\end{align*}
Moreover, the locality of $f$ yields 
\begin{equation}\label{e.entropy4}
\begin{aligned}
\sum_{y\in \Z^d} \la  u_{t}(y, \eta)\tilde h^M(\eta) \, | \, y  \ra_\eta \les \sum_{y\in B_{r_0}}\la  P_t\tilde h^M(\eta)\, \big| \, y  \ra_\eta
\end{aligned}
\end{equation}
where $B_{r_0}$ denotes the support of $f$.
We now observe that, since $\tilde h^M$ does not depend on the tagged particle $X$, it holds that $\L \tilde h^M = \L_K \tilde h^M$, and therefore also that $P_t \tilde h^M = \tilde P_t \tilde h^M$, where $\tilde P_t:=e^{t\L_K}$. Therefore, the contractivity of the
$L^1$-norm for the Kawasaki dynamics yields that
\begin{align*}
\langle P_t\tilde h^M \, | \, y  \rangle_\eta = \langle \tilde P_t\tilde h^M \, | \, y \rangle_\eta  = \rho^{-1} \la \tilde P_t\tilde h^M  \1_{\{\eta(y) = 1\}}\ra \leq \rho^{-1} \langle \tilde P_t\tilde h^M \, \rangle \leq \rho^{-1} \langle \tilde h^M \rangle = \rho^{-1}. 
\end{align*}
Combining this last inequality with \eqref{e.entropy4} and \eqref{e.entropy3} yields
\begin{align*}
\sum_{x \in B_L} \la \bigl(\pi^{\ell}_{L} u_{t}(\M^{\ell}_{L},x) \bigr)^{2p}   \1^\rho(\M_L^\ell)  \,|\,   x \ra \lesssim  |B_\ell|^{-2p+1}.
\end{align*}
Our choice of $\ell= \lfloor \delta \sqrt{t} \rfloor \vee 1$ allows us to conclude the proof of Proposition~\ref{p.entropy}.
\end{proof}

We can now conclude the proof of Proposition~\ref{prop.goal}, and therefore also of Theorems~\ref{HKbound} and \ref{t.main}.

\begin{proof}[Proof of Proposition~\ref{prop.goal}]
Proposition~\ref{p.entropy} applied to the second term on the right side of \eqref{e.intermediate} yields
\begin{align*}
|| u_t ||_{2p}^{2p} & \lesssim \delta^2 \, t\sum_{x\in \Z^d}\sum_{e\in\B} \la a_e( (u_{t}^e)^p - (u_t)^p)^2 \, | \, x \ra
+ C(\delta) t^{\frac d 2 (1 - 2p)},
\end{align*}
as desired.
\end{proof}

\appendix

\section{Technical lemmas}

\begin{lemma}
\label{l.counting}
There exists a constant $C=C(\pi,e)$ such that for any $N,N_1,N_2\in \Z_{\geq 1}$ with $N>N_1>N_2>0$, we have 
\[
 \biggl[\binom{N}{N_1 }\biggr]^{-1}\binom{N-N_2}{N_1-N_2}\leq C\sqrt{\frac{N_1(N-N_2)}{N(N_1-N_2)}} \left(\frac{N_1}{N}\right)^{N_2}.
 \]
\end{lemma}

\begin{proof}
Defining $\rho:=N_1/N,x=N_2/N$, we apply Stirling's formula to derive
\[
\begin{aligned}
\biggl[\binom{N}{N_1 }\biggr]^{-1}\binom{N-N_2}{N_1-N_2}=\frac{(N-N_2)!(N_1)!}{(N_1-N_2)!N!}\leq &C\sqrt{\frac{N_1(N-N_2)}{N(N_1-N_2)}} \left[\frac{(1-x)^{1-x}\rho^{\rho}}{(\rho-x)^{\rho-x}}\right]^N\\
=&C\sqrt{\frac{N_1(N-N_2)}{N(N_1-N_2)}}  e^{Nf_\rho(x)},%\frac{(N-N_2)^{N-N_2+\frac12}e^{-(N-N_2)}}{}
\end{aligned}
\]
where
\[
f_\rho(x):=\rho\log\rho+(1-x)\log(1-x)-(\rho-x)\log(\rho-x), \  x\in [0,\rho).
\]
A straightforward calculation gives $f_\rho(0)=0$, $f'_\rho(0)=\log\rho$, and $f''_\rho(x)<0$, thus 
\[
f_\rho(x)\leq x\log\rho.
\]
The proof is complete.
\end{proof}

\begin{lemma}
\label{l.exma}
Fix any $x\in \Z^d$, let $\xi(z)=|z-x|$, and define the martingale 
\[
M_t = \xi(\X_t) - \xi(\X_0) - \int_0^t L\xi(\X_s,\bfeta_s) \, \d s.
\] 
For every $(x,\eta) \in \Omega$ and $t,\lambda > 0$, we have
\begin{equation*} 
\EE_{(x,\eta)} \Ll[ \exp(\lambda M_t) \Rr] \le \exp \Ll( 2d  \Ll(e^\lambda - 1 - \lambda\Rr) t \Rr) .
\end{equation*}
\end{lemma}

\begin{proof}

The proof is inspired by \cite{freedman}. We fix $\lambda \ge 0$, $e(\lambda) := e^\lambda - 1- \lambda$, and show that the process $(E_t)_{t \ge 0}$ defined by
\begin{equation*}  \label{e.}
E_t := \exp \Ll( \lambda M_t - e(\lambda) \langle M \rangle_t \Rr) 
\end{equation*}
is a supermartingale under $\PP_{(x,\eta)}$, where $(\la M \ra_t)_{t \ge 0}$ denotes the predictable quadratic variation of $M$. The conclusion then follows since $\E[E_t] \le \E[E_0] =1$ and $\la M \ra_t \le 2dt$.

We write $M_{t-}$ to denote the left limit of $M$ at time $t$, and $\Delta M_t := M_t - M_{t-} $ to denote the size of the jump at time $t$. The key ingredient of the argument is that 
\begin{equation}
\label{e.jumps}
\sup_t \Delta M_t \le 1.
\end{equation}
We denote by $([M]_t)_{t \ge 0}$ the bracket process associated with $M$. Since $M$ is of bounded variation, this is simply
\begin{equation}  \label{e.bracket}
[M]_t := \sum_{0 \le s \le t} (\Delta M_t)^2.
\end{equation}
By an extension of the fundamental theorem of calculus that allows for jumps, see e.g.\ \cite[Theorem II.7.31]{protter}, we have, for every $s \le t$, 
\begin{equation}  \label{e.calculus}
E_t - E_s = \int_s^t \lambda E_{r-} \, \d M_r - \int_s^t e(\lambda) E_{r-} \d \la M \ra_r + \sum_{0 \le s \le t} \Ll(\Delta E_r - \lambda E_{r-} \Delta M_r\Rr).
\end{equation}
By \cite[Corollary~3.2]{freedman}, we have
\begin{equation*}  \label{e.}
x \le 1 \quad \implies \quad e^{\lambda x} \le 1 + \lambda x + e(\lambda) x^2.
\end{equation*}
By \eqref{e.jumps}, we deduce that
\begin{equation*}  \label{e.}
\Delta E_r = E_{r-} \Ll( e^{\lambda \Delta M_r} -1\Rr)\le E_{r-} \Ll( \lambda  \Delta M_r + e(\lambda) (\Delta M_r)^2\Rr) .
\end{equation*}
Combining this with \eqref{e.bracket} and \eqref{e.calculus}, we obtain
\begin{equation*}  \label{e.}
E_t - E_s \le \int_s^t \lambda E_{r-} \, \d M_r + \int_s^t e(\lambda) E_{r-} \d \Ll([M]-\la M \ra\Rr)_r.
\end{equation*}
By \cite[Proposition~4.50]{jacshi}, the process $([M]_t - \la M \ra_t)_{t \ge 0}$ is a martingale. The proof is therefore complete.
\end{proof}

\begin{lemma}
\label{l.boundary}
There exists $C(\rho) < \infty$ such that for every $k\in \Z_{\geq 1}$ and $\beta>1$, we have  
\begin{equation}
\begin{aligned}
&\sum_{x\in\Z^d}\sum_{e\in\partial \B_k}\1_{\{x\notin e\}}\la (\cA_k h)^e-\cA_k h, a_e(h^e-h)\, |\, x\ra \\
&\leq C\beta \sum_{x\in\Z^d}\sum_{e\in \partial \B_k} \la a_e(h^e-h)^2\, |\, x\ra+\frac{C}{\beta} \sum_{x\in\Z^d}\left[\la (\cA_{k+1}h)^2\, |\, x\ra-\la (\cA_{k}h)^2\, |\, x\ra\right].
\label{e.bddi}
\end{aligned}
\end{equation}
\end{lemma}

\begin{proof}
For any $e\in \partial \B_k$, we write $e=(y,z)$ with $y\in B_k, z\notin B_k$. We first show that~\eqref{e.bddi} holds when $h$ only depends on $\eta(y),\eta(z)$, and then consider the general case.

Since $x\notin e$, to simplify the notation we just write $h=h(\eta(y),\eta(z))$. Recalling that $\eta(y),\eta(z)$ are independent Bernoulli random variables with parameter $\rho$, we have
\[
\begin{aligned}
\cA_kh=h(\eta(y),1)\rho+h(\eta(y),0)(1-\rho),\\
(\cA_kh)^e=h(\eta(z),1)\rho+h(\eta(z),0)(1-\rho).
\end{aligned}
\]
Since $a_e(h^e-h)=\1_{\{\eta(y)\neq \eta(z)\}}[h(\eta(z),\eta(y))-h(\eta(y),\eta(z))]$, we further obtain
\[
\begin{aligned}
&\la (\cA_k h)^e-\cA_k h, a_e(h^e-h)\, |\, x\ra\\
&=\rho(1-\rho)[h(0,1)\rho+h(0,0)(1-\rho)-h(1,1)\rho-h(1,0)(1-\rho)][h(0,1)-h(1,0)]\\
&+\rho(1-\rho)[h(1,1)\rho+h(1,0)(1-\rho)-h(0,1)\rho-h(0,0)(1-\rho)][h(1,0)-h(0,1)].
\end{aligned}
\]
Thus, 
\begin{equation}
\begin{aligned}
&\la (\cA_k h)^e-\cA_k h, a_e(h^e-h)\, |\, x\ra\\
&\leq C\beta|h(0,1)-h(1,0)|^2+\frac{C}{\beta}|h(0,1)-h(1,1)|^2+\frac{C}{\beta}|h(0,0)-h(1,0)|^2
\end{aligned}
\label{e.a1b1}
\end{equation}
for some $C=C(\rho)$ and any $\beta>0$. 

By a similar calculation, we have 
\[
\la a_e(h^e-h)^2\, |\, x\ra=2\rho(1-\rho)|h(0,1)-h(1,0)|^2,
\]
and 
\[
\begin{aligned}
&\la (\cA_{k+1}h)^2\, |\, x\ra-\la (\cA_{k}h)^2\, |\, x\ra\\
&= \rho^2(1-\rho)|h(1,1)-h(1,0)|^2+\rho(1-\rho)^2 |h(0,1)-h(0,0)|^2.
\end{aligned}
\]
It is clear that the first term on the right side of \eqref{e.a1b1} can be controlled by $\la a_e(h^e-h)^2\, |\, x\ra$, and the last two terms can be controlled by 
\[
\la a_e(h^e-h)^2\, |\, x\ra+\la (\cA_{k+1}h)^2\, |\, x\ra-\la (\cA_{k}h)^2\, |\, x\ra
\]
 after applying the triangle inequality. Thus  \eqref{e.bddi} is proved when $h$ only depends on $\eta(y),\eta(z)$.

Now we consider the general case. Fix $x\in\Z^d$, and for any $e=(y,z)\in \partial \B_k$ with $z\notin B_k$, we define $h_z:=\la h\, |\, \mcl F(B_k\cup \{z\})\ra_x$, where to avoid confusion, we use $\la\,\cdot\,\ra_x$ to denote $\la\,\cdot\,|\,x\ra$, then we have
\[
\la (\cA_k h)^e-\cA_k h, a_e(h^e-h)\, |\, x\ra=\la (\cA_k h_z)^e-\cA_k h_z, a_e(h_z^e-h_z)\, |\, x\ra.
\]
For each realization of $\{\eta(\tilde{y}): \tilde{y}\in B_k\setminus\{y\}\}$, we view $h_z$ as a function of $\eta(y),\eta(z)$, and the previous discussion shows that 
\[
\begin{aligned}
&\la (\cA_k h_z)^e-\cA_k h_z, a_e(h_z^e-h_z)\, |\, x\ra \\
&\leq C\beta\la a_e(h_z^e-h_z)^2\, |\, x\ra+\frac{C}{\beta} \left[\la (\cA_{k+1}h_z)^2\, |\, x\ra-\la (\cA_{k}h_z)^2\, |\, x\ra\right].
\end{aligned}
\]
For the first term on the right side of the above inequality, we have 
\[
\la a_e(h_z^e-h_z)^2\, |\, x\ra\leq \la a_e(h^e-h)^2\, |\, x\ra,
\]
so it remains to show 
\begin{equation}
\sum_{e\in \partial \B_k} \left(\la (\cA_{k+1}h_z)^2\, |\, x\ra-\la (\cA_{k}h_z)^2\, |\, x\ra\right)\leq \la (\cA_{k+1}h)^2\, |\, x\ra-\la (\cA_{k}h)^2\, |\, x\ra.
\label{e.a1b2}
\end{equation}
Let $B_{k+1}\setminus B_k=\{z_i\}_{i=1}^N$ and $\mcl F_l=\mcl F(B_k\cup\{z_i\}_{i=1}^l), l=1,\ldots,N$. We have
\[
\la (\cA_{k+1}h)^2\, |\, x\ra-\la (\cA_{k}h)^2\, |\, x\ra=\sum_{l=1}^N \la \la h\, |\, \mcl F_l\ra^2-\la h\, |\, \mcl F_{l-1}\ra^2 \, \big|\, x\ra.
\]
For any $e=(y,z)\in \partial \B_k$, it is clear that $z=z_l$ for some $l=1,\ldots,N$. We claim that 
\begin{equation}
\la (\cA_{k+1}h_{z_l})^2\, |\, x\ra-\la (\cA_{k}h_{z_l})^2\, |\, x\ra\leq \la \la h\, |\, \mcl F_l\ra^2-\la h\, |\, \mcl F_{l-1}\ra^2 \, \big|\, x\ra,
\label{e.a1b3}
\end{equation}
which implies \eqref{e.a1b2} and completes the proof. To prove \eqref{e.a1b3}, we observe that $\cA_{k+1}h_{z_l}$ (resp. $\cA_k h_{z_l}$) is the average of $\la h\, |\, \mcl F_l\ra$ (resp. $\la h\,|\, \mcl F_{l-1}\ra$) with respect to $\{z_i\}_{i=1}^{l-1}$, thus 
\[
\la (\cA_{k+1} h_{z_l}-\cA_k h_{z_l})^2 \, |\, x\ra\leq \la ( \la h\, |\, \mcl F_l\ra-\la h\, |\, \mcl F_{l-1}\ra)^2\,|\, x\ra,
\]
which reduces to \eqref{e.a1b3} by the property of conditional expectation.
\end{proof}

\begin{lemma}
\label{l.moment}
Let $\E$ denote an expectation on a probability space, and $p \ge 1$. There exists a constant $C(p)< \infty$ such that for every random variable $X \ge 0$,
\begin{equation*}  
\mathbb{E} \Ll[|X - \mathbb{E}[X]|^{2p}\Rr] \le C \mathbb{E} \Ll[ \Ll( X^p - \E[X^p] \Rr)^2  \Rr] .
\end{equation*}
\end{lemma}
\begin{proof}
Let $Y$ be an independent copy of $X$. We have
$$
\|X - \mathbb{E}[X]\|_{2p}  = \|\mathbb{E}\left[X - Y \, | \, X\right]\|_{2p}  \le \|X - Y\|_{2p}.
$$
Moreover, there exists a constant $C(p)$ such that for every $x,y \ge 0$,
$$
|x-y|^p \le C(p) |x^p - y^p|.
$$
Indeed, it suffices to verify this for $x = 1$ and $y \in [0,1]$ by homogeneity and symmetry. This is then a simple exercise. As a consequence, we deduce
$$
\|X - \mathbb{E}[X]\|_{2p} \le C(p) \||X^p - Y^p|^{1/p}\|_{2p} = C(p) \|X^p - Y^p\|_{2}^{1/p}.
$$
We conclude by the triangle inequality.
\end{proof}

\subsection*{Acknowledgments} We thank Felix Otto for stimulating discussions, and Mikael de la Salle for showing us the beautiful proof of Lemma~\ref{l.moment} reproduced here. Y.G.\ was partially supported by the NSF through DMS-1613301.

\newcommand{\noop}[1]{} \def\cprime{$'$}

%\bibliographystyle{abbrv}
%\bibliography{ips}

\end{document}